\documentclass[12pt]{amsart}
\usepackage{amssymb}
\usepackage{eucal}
\usepackage{bm}
\usepackage[all,cmtip]{xy}
\usepackage{color}
\usepackage[bookmarks=false]{hyperref}
\usepackage{multirow,bigdelim}
\usepackage{url}
\usepackage[titletoc]{appendix}
\usepackage{lipsum}
\usepackage{xtab}

\allowdisplaybreaks

\newtheorem{theorem}[equation]{Theorem}
\newtheorem{lemma}[equation]{Lemma}
\newtheorem{proposition}[equation]{Proposition}

\newtheorem{conjecture}[equation]{Conjecture}

\theoremstyle{definition}
\newtheorem{definition}[equation]{Definition}
\newtheorem{tablesymb}[equation]{Table of Symbols}

\theoremstyle{remark}
\newtheorem{remark}[equation]{Remark}

\numberwithin{equation}{subsection}

\allowdisplaybreaks[1]

\newcommand{\DD}{\mathbb{D}}
\newcommand{\FF}{\mathbb{F}}
\newcommand{\ZZ}{\mathbb{Z}}
\newcommand{\QQ}{\mathbb{Q}}
\newcommand{\RR}{\mathbb{R}}

\newcommand{\GG}{\mathbb{G}}

\newcommand{\CC}{\mathbb{C}}

\newcommand{\NN}{\mathbb{N}}

\newcommand{\Fq}{\mathbb{F}_{q}}

\newcommand{\bn}{\mathbf{n}}

\newcommand{\bx}{\mathbf{x}}
\newcommand{\bX}{\mathbf{X}}
\newcommand{\bu}{\mathbf{u}}
\newcommand{\bv}{\mathbf{v}}

\newcommand{\bC}{\mathbf{C}}

\newcommand{\bz}{\mathbf{z}}

\newcommand{\bw}{\mathbf{w}}

\newcommand{\cD}{\mathcal{D}}

\newcommand{\cO}{\mathcal{O}}

\newcommand{\cZ}{\mathcal{Z}}

\DeclareMathAlphabet{\matheur}{U}{eur}{m}{n}

\newcommand{\fs}{\mathfrak{s}}
\newcommand{\fm}{\mathfrak{m}}

\newcommand{\fn}{\mathfrak{n}}

 \DeclareMathOperator{\Lie}{Lie}
 
\DeclareMathOperator{\Mat}{Mat} 
\DeclareMathOperator{\End}{End}

 \DeclareMathOperator{\wt}{wt}
\DeclareMathOperator{\Li}{Li}

\DeclareMathOperator{\dep}{dep}

\DeclareMathOperator{\lcm}{lcm}

\newcommand{\ok}{\overline{k}}

\newcommand{\tr}{\mathrm{tr}}

\newcommand{\Lis}{\Li^{\star}}

\newcommand{\DDConv}{\DD^{\text{\fontsize{6pt}{0pt}\selectfont \rm{Conv}}}}
\newcommand{\DDDef}{\DD^{\text{\fontsize{6pt}{0pt}\selectfont \rm{Def}}}}
\newcommand{\ocLConv}{\overline{\mathcal{L}}^{\text{\fontsize{6pt}{0pt}\selectfont \rm{Conv}}}}
\newcommand{\ocLDef}{\overline{\mathcal{L}}^{\text{\fontsize{6pt}{0pt}\selectfont \rm{Def}}}}

\newcommand{\power}[2]{{#1 [\![ #2 ]\!]}}
\newcommand{\laurent}[2]{{#1 (\!( #2 )\!)}}

\newcommand{\hcmspl}[1]{\mathfrak{h}_{\mathrm{CMSPL}}^{#1}}
\newcommand{\bst}{\star }
\newcommand{\cLis}{\mathcal{L}i^{\star}}
\newcommand{\cLisleq}[1]{\mathcal{L}i^{\star}_{\leq #1}}

\definecolor{ForestGreen}{rgb}{0.0, 0.5, 0.0}

\newcommand{\C}{\ensuremath \mathbb{C}}

\newcommand{\F}{\ensuremath \mathbb{F}}

\newcommand{\fa}{\mathfrak{a}}

\def\XXint#1#2#3{{\setbox0=\hbox{$#1{#2#3}{\int}$ }
\vcenter{\hbox{$#2#3$ }}\kern-.6\wd0}}

\title[Algebra structure of multiple zeta values in positive characteristic]{Algebra structure of multiple zeta values in positive characteristic}

\author{Chieh-Yu Chang}
\address{Department of Mathematics, National Tsing Hua University, Hsinchu City 30042, Taiwan
  R.O.C.}
\email{cychang@math.nthu.edu.tw}

\author{Yen-Tsung Chen}
\address{Department of Mathematics, National Tsing Hua University, Hsinchu City 30042, Taiwan
  R.O.C.}
\email{s107021901@m107.nthu.edu.tw}

\author{Yoshinori Mishiba}
\address{Department of Mathematical Sciences, University of the Ryukyus, 1 Senbaru, Nishihara-cho, Okinawa 903-0213,  Japan}
\email{mishiba@sci.u-ryukyu.ac.jp}


\date{\today}

\begin{document}

\begin{abstract}
This paper is a culmination of \cite{CM19b} on the study of multiple zeta values (MZV's) over function fields in positive characteristic.  For any finite place $v$ of the rational function field $k$ over a finite field, we prove that the $v$-adic MZV's satisfy the same $\bar{k}$-algebraic relations that their corresponding $\infty$-adic MZV's satisfy. Equivalently, we show that the $v$-adic MZV's form an algebra with multiplication law given by the $q$-shuffle product which comes from the $\infty$-adic MZV's, and there is a well-defined $\bar{k}$-algebra
homomorphism from the $\infty$-adic MZV's to the $v$-adic MZV's.\end{abstract}


\date{\today}
\maketitle

\section{Introduction}
\subsection{Classical conjecture}
Let $\NN$ be the set of positive integers. For a positive integer $r$, an $r$-tuple $\fs=(s_{1},\ldots,s_{r})\in \NN^{r}$ is called an index, and called {\it{admissible}} if $s_{1}>1$. We put $\wt(\fs):=\sum_{i=1}^{r}s_{i}$ and $\dep(\fs):=r$. Classical real-valued multiple zeta values (abbreviated as MZV's) are generalizations of special values of the Riemann zeta function at positive integers at least $2$. The MZV at an admissible index $\fs=(s_{1},\ldots,s_{r})$ is defined by the following series \[\zeta(\fs):=\sum_{n_{1}>n_{2}>\cdots>n_{r}\geq 1}\frac{ 1  }{n_{1}^{s_{1}}\cdots n_{r}^{s_{r}}  } \in \RR^{\times}.\] The weight and depth of the presentation $\zeta(\fs)$ are defined by $\wt(\fs)$ and $\dep(\fs)$ respectively.

MZV's have deep properties and have appeared in recent decades in connection with various topics including Grothendieck-Teichm\"uller groups, Drinfeld associators and KZ equations and periods of mixed Tate motives etc (see~\cite{An04, Br12, BGF19, DG05, Dr90, F11, Gon02, R02, Te02, Zh16}). One of the core problems on the topic of MZV's is to study their algebraic relations, and how to generate $\QQ$-linear relations among them has been well-developed. For instance,  the machinery of regularized double shuffle relations~\cite{IKZ06} produces rich $\QQ$-linear relations among MZV's of the same weight.

Let $p$ be a prime number.  In the parallel but extremely different world, namely the $p$-adic field, Furusho~\cite{F04} defined $p$-adic MZV's. The starting point is that for an admissible index $\fs=(s_{1},\ldots,s_{r})\in \NN^{r}$, the MZV $\zeta(\fs)$ is the limit of the one-variable multiple polylogarithm 
\[\Li_{\fs}(z) :=\sum_{n_{1}>n_{2}>\cdots>n_{r}\geq 1}\frac{ z^{n_{1}}  }{n_{1}^{s_{1}}\cdots n_{r}^{s_{r}}  }   \]
for $|z|<1, z\rightarrow 1$.  Furusho considered the one-variable $p$-adic multiple polylogarithm $\Li_\fs(z)_p$, which is the same power series as $\Li_{\fs}(z)$, but treated $p$-adically. He then made an analytic continuation of $\Li_{\fs}(z)_{p}$ by Coleman's p-adic iterated integration theory and then defined the p-adic MZV $\zeta_{p}(\fs)$  to be certain limit value at $1$ of analytically continued function of $\Li_{\fs}(z)_{p}$. Related details are referred to Furusho's paper~\cite{F04}.
The weight and depth of the presentation of the $p$-adic MZV $\zeta_{p}(\fs)$ are defined to be $\wt(\fs)$ and $\dep(\fs)$
respectively. 

Note that in the case of depth one, Furusho's p-adic zeta value $\zeta_{p}(s)$ equals the Kubota-Leopoldt  $p$-adic zeta value at $s$ up to a scalar multiplication by
$(1-p^{-s})^{-1}$. In particular, we have $\zeta_{p}(2n)=0$ for $n\in \NN$.  As Kubota-Leopoldt $p$-adic zeta function {\it $p$-adically interpolates} the
special values of Riemann zeta function at negative inetgers, one can ask the natural question: what kind of spark can these two seemingly similar values,  real-valued MZV's and $p$-adic MZV's, but living in completely different worlds have? The following fundamental conjecture gives an explicit connection between these two kinds of MZV's.

\begin{conjecture}\label{Conj: Conj1-Intro}
For any prime number $p$, the $p$-adic MZV's satisfy the same $\QQ$-algebraic relations that their corresponding real-valued MZV's satisfy. That is, if \[f(\zeta(\fs_{1}),\ldots,\zeta(\fs_{m}))=0 \] for $f\in\QQ[X_{1},\ldots,X_{m}]$, then we have \[f(\zeta_{p}(\fs_{1}),\ldots,\zeta_{p}(\fs_{m}))=0.\]
\end{conjecture}

Let $\mathfrak{Z}$ (resp.~$\mathfrak{Z}_{p}$) be the $\QQ$-vector space spanned by $1$ and all real-valued MZV's (resp.~by $1$ and all $p$-adic MZV's). It is well-known that $\mathfrak{Z}$ forms a $\QQ$-algebra with two multiplication laws given by shuffle product and stuffle product \cite{R02, IKZ06, BGF19}. By~\cite{F04, BF06}, one also knows that $\mathfrak{Z}_{p}$ forms a $\QQ$-algebra with two multiplication laws given by shuffle product and stuffle product such as the case of real-valued MZV's. Therefore, the conjecture above is equivalent to the following one.

\begin{conjecture}\label{Conj: Conj2-Intro} For any prime number $p$, the following map
\[ \phi_{p}:=\left( \zeta(\fs)\mapsto \zeta_{p}(\fs) \right): \mathfrak{Z}\twoheadrightarrow  \mathfrak{Z}_{p}\]
is a well-defined $\QQ$-algebra homomorphism. 

\end{conjecture}

There are several ways to illustrate the conjectures above.
\begin{enumerate}
\item Ihara, Kaneko and Zagier~\cite{IKZ06} gave a conjecture asserting that the regularized double shuffle relations generate all $\QQ$-algebraic relations among the real-valued MZV's. Furusho-Jafari~\cite{FJ07} showed that the $p$-adic MZV's satisfy the regularized double shuffle relations. It follows that combining Ihara-Kaneko-Zagier conjecture and Furusho-Jafari's result would imply Conjecture~\ref{Conj: Conj1-Intro}.

\item 
Clues of the formulation of the conjecture above also come from \cite{F06, F07}. For an integer $n\geq 2$, we let $\mathfrak{Z}_{n}$ (resp.~$\mathfrak{Z}_{n,p}$) be the $\QQ$-vector space spanned by real-valued MZV's of weight $n$ (resp.~$p$-adic MZV's of weight $n$). Considering the graded algebra $Z:=\mathbb{Q}\oplus\bigoplus_{n\geq 2}\mathfrak{Z}_{n}$ (resp. $Z_p:=\mathbb{Q}\oplus\bigoplus_{n\geq 2}\mathfrak{Z_{n,p}}$), Furusho \cite[Conj.~A]{F06} conjectured that $\mathcal{O}({\underline{GRT}}_1)$ is isomorphic to $Z/(\pi^2)$ and in \cite[Sec.~3.1]{F07}  he explained that
there is a surjection from $\mathcal{O}({\underline{GRT}}_1)$ to $Z_p$. Here ${\underline{GRT}}_1$ is the unipotent part of the graded Grothendieck-Teichm\"uller group ${\underline{GRT}}$, which is a pro-algebraic group over $\mathbb{Q}$. For more details, see~\cite{F06, F07}. On the other hand, Goncharov's direct sum conjecture~\cite{Gon97} for MZV's asserts that $\mathfrak{Z}=Z$, and on the $p$-adic side we have a natural surjective $\QQ$-algebra homomorphism
\[ Z_{p}\twoheadrightarrow \mathfrak{Z}_{p} .\]
So, conjecturally the composite map
\[ \mathfrak{Z}=Z \twoheadrightarrow Z/(\pi^{2}) \simeq \mathcal{O}({\underline{GRT}}_1) \twoheadrightarrow Z_{p} \twoheadrightarrow \mathfrak{Z}_{p} \]
gives rise to a surjective $\QQ$-algebra homomorphism from $\mathfrak{Z}$ to $\mathfrak{Z}_{p}$. 
\end{enumerate}

Moreover, there is a motivic illustration for a conjectural surjective homomorphism $\mathfrak{Z}\twoheadrightarrow\mathfrak{Z}_{p}$, which the authors learned from F.~Brown's talk on \lq\lq motivic periods and applications \rq\rq  in Hausdorff Research Institute for Mathematics in 2018. First, Deligne also defined $p$-adic MZV's and Furusho showed in~\cite{F07} that Deligne's $p$-adic MZV's generate the same space $\mathfrak{Z}_{p}$. Let $\mathfrak{Z}^{\mathrm{mot}}$ be the $\QQ$-algebra of motivic MZV's, and one knows that there is a $\QQ$-algebra homomorphism (see \cite{Gon02} and \cite{Br12}) addressed as the period map
\[ \mathfrak{Z}^{\mathrm{mot}} \twoheadrightarrow \mathfrak{Z}, \]
and the Grothendick periods conjecture for MZV's predicts that this is an isomorphism. On the other hand,  from the $p$-adic period map one has a $\mathbb{Q}$-algebra homomorphism (cf.~\cite[(3.11)]{F07})
\[ \mathfrak{Z}^{\mathrm{mot}}/(\zeta^{\mathrm{mot}}(2)) \twoheadrightarrow \mathfrak{Z}_{p} ,\]
and so conjecturally there is a $\mathbb{Q}$-algebra homomorphism 
\[\mathfrak{Z}\twoheadrightarrow \mathfrak{Z}_{p} .\]

The aim of this paper is to prove the precise analogue of Conjecture~\ref{Conj: Conj1-Intro} in the setting of function fields in positive characteristic. Note that our methods of proof are through logarithms of $t$-modules, which are entirely different from the above points of view in the characteristic zero case.

\subsection{The main result} Let $q$ be a power of a prime number $p$, and let $\FF_{q}$ be a finite field of $q$ elements.  Let $A:=\FF_{q}[\theta]$ be the polynomial ring with quotient field $k:=\FF_{q}(\theta)$. We let $k_{\infty}$ be the completion of $k$ at the infinite place $\infty$, and $\CC_{\infty}$ be the $\infty$-adic completion of a fixed algebraic closure of $k_{\infty}$. We let $\bar{k}$ be the algebraic closure of $k$ inside $\CC_{\infty}$. 
 
 The $\infty$-adic multiple zeta values are defined by Thakur~\cite{T04}: for any index $\fs=(s_{1},\ldots,s_{r})\in \NN^{r}$, we define
\[
\zeta_{A}(\fs):=\sum \frac{1}{a_{1}^{s_1}\cdots a_{r}^{s_r}}\in k_{\infty},
\] where the sum is over all monic polynomials $a_{1},\ldots,a_{r}$ in $A$ with the restriction $\deg_{\theta}a_{1}> \deg_{\theta}a_{2}>\cdots>\deg_{\theta}a_{r}$. For $r=1$, the values above were introduced by Carlitz~\cite{Ca35} and called Carlitz zeta values. We call $\wt(\fs):=\sum_{i=1}^{r}s_{i}$ the weight and $r:=\dep(\fs)$ the depth of the presentation $\zeta_{A}(\fs)$. In~\cite{T10}, Thakur showed that for any two indices $\fs \in \NN^{r} $ and $\fs'\in \NN^{r'}$, one has 
\begin{equation}\label{E:sum-shuffle}
 \zeta_{A}(\fs)\cdot \zeta_{A}(\fs')=\sum_{j}f_{j} \zeta_{A}(\fs_{j})
\end{equation}
for some finitely many $f_{j}\in \FF_{p}$ and $\fs_{j}\in \NN^{\dep(\fs_{j})}$ depending on $q$ with $\wt(\fs_{j})=\wt(\fs)+\wt(\fs')$ and $\dep(\fs_{j})\leq \dep(\fs)+\dep(\fs')$, where $\FF_{p}$ is the prime field of $k$. We simply call (\ref{E:sum-shuffle}) the $q$-shuffle relations (or $q$-shuffle product), which Thakur called sum-shuffle relations. Note that in our positive characteristic setting, the $q$-shuffle product is neither the classical shuffle product nor stuffle product (see H.-J. Chen's explicit formula~\eqref{E: H-J formula}). Because of the $q$-shuffle product, the $\infty$-adic MZV's form an $\FF_{p}$-algebra.

Given a monic irreducible polynomial $v$ of $A$, we let $k_{v}$ be the completion of $k$ at $v$ and let $\CC_{v}$ be the $v$-adic completion of a fixed algebraic closure of $k_{v}$. Throughout this article, we always fix an embedding $\bar{k} \hookrightarrow \CC_{v}$ over $k$ once a finite place $v$ is given. Based on the formula of $\infty$-adic MZV's in terms of Carlitz multiple polylogarithms (abbreviated as CMPL's) established in~\cite{C14}, the first and third authors of the present paper introduced the Carlitz multiple star polylogarithms (abbreviated as CMSPL's) given in \eqref{CMSPL} and derived the formula of $\infty$-adic MZV's as $k$-linear combinations of CMSPL's at integral points in \eqref{E: MZV-formula}. In the depth one case, CMSPL's are reduced to Carlitz polylogarithms and such formula was established by Anderson-Thakur~\cite{AT90}.

Inspired by Furusho's strategy for defining $p$-adic MZV's~\cite{F04}, for any index $\fs\in \NN^{r}$ the first and third authors treated CMSPL's for $v$-adic convergence in~\cite{CM19a} and used action of certain $t$-modules for which $v$-adic CMSPL's can be extended to be defined at integral points. Then they used the same formula of $\infty$-adic MZV's~\eqref{E: MZV-formula} to define the $v$-adic MZV  $\zeta_{A}(\fs)_{v}$ in~\eqref{E: v-adic MZV-formula} for any index $\fs$.    As same as the case of $\infty$-adic MZV's, the weight and the depth of the presentation $\zeta_{A}(\fs)_{v}$ are defined to be $\wt(\fs)$ and $\dep(\fs)$ respectively. Note that Thakur~\cite{T04} also defined $v$-adic MZV's but his definition is different from ours.

In~\cite{Go79}, Goss defined a $v$-adic zeta function that interpolates Carlitz zeta values at non-positive integers and obtained $v$-adic zeta values at positive integers, which are simply called {\it{ Goss' $v$-adic zeta values}}, which are equal to Thakur's $v$-adic MZV's of depth one. In the depth one case, our $v$-adic zeta value $\zeta_{A}(s)_{v}$ is identical to Goss' $v$-adic zeta value~\cite{Go79} at $s$ multiplied by $(1-v^{-s})^{-1}$ (see \cite[Thm.~3.8.3. (II)]{AT90}), and so $\zeta_{A}(n)_{v}=0$ for all positive integers $n$ divisible by $q-1$ by the work of Goss.  This phenomenon is parallel to the $p$-adic case mentioned above.

Let $\overline{\cZ} \subset \CC_{\infty}$ (resp.\ $\overline{\mathcal{Z}}_{v} \subset \CC_{v}$) be the $\ok$-vector space spanned by $1$ and all $\infty$-adic MZV's (resp.~$1$ and all $v$-adic MZV's). It is shown in~\cite[Cor.~6.4.3]{CM19b} that the map $\overline{\mathcal{Z}}\twoheadrightarrow \overline{\mathcal{Z}}_{v}$ given by $\zeta_{A}(\fs)\mapsto \zeta_{A}(\fs)_{v}$ is a well-defined $\ok$-linear map. Our main theorem stated below is a function field analogue of Conjecture~\ref{Conj: Conj1-Intro} but it is in stronger form as it is over algebraic coefficients. 
\begin{theorem}\label{T: IntrodT1}
For any finite place $v$ of $k$, the $v$-adic MZV's satisfy the same $\bar{k}$-algebraic relations that their corresponding $\infty$-adic MZV's satisfy. That is, if \[g(\zeta_{A}(\fs_{1}),\ldots,\zeta_{A}(\fs_{m}))=0\] for $g\in \bar{k}[X_{1},\ldots,X_{m}]$, then we have \[g(\zeta_{A}(\fs_{1})_{v},\ldots,\zeta_{A}(\fs_{m})_{v})=0.\]
\end{theorem}
 
Note that $\overline{\mathcal{Z}}$ forms a $\ok$-algebra because of \eqref{E:sum-shuffle}. The theorem above is equivalent to the following.

\begin{theorem}\label{T: IntrodT2}
Let $v$ be a finite place of $k$. Then the following hold.
\begin{enumerate}
\item $v$-adic MZV's satisfy the $q$-shuffle relations in the sense that
\[   \zeta_{A}(\fs)_{v}\cdot \zeta_{A}(\fs')_{v}=\sum_{j}f_{j} \zeta_{A}(\fs_{j})_{v} \]
with notation given in~\eqref{E:sum-shuffle}.
\item $\overline{\cZ}_{v}$ forms a $\ok$-algebra and the following map $\overline{\mathcal{Z}}\twoheadrightarrow  \overline{\mathcal{Z}}_{v}$ given by  $\zeta_{A}(\fs)\mapsto \zeta_{A}(\fs)_{v}$ is a well-defined $\bar{k}$-algebra homomorphism. In particular, the kernel contains the principal ideal generated by $\zeta_{A}(q-1)$.
\end{enumerate}
\end{theorem}

The theorem above gives an affirmative answer of part of the questions in~\cite[Rem.~6.4.4]{CM19b}, which arose from numerical evidence using Sage. In Section~\ref{Subsec: Example} we give an example for computing the product of the $v$-adic single zeta value $\zeta_{v}(1)$ with itself  for a very special $q$. As can be seen from that example, a direct calculation proof of Theorem~\ref{T: IntrodT2} is impracticable because the definition of $\zeta_{A}(\fs)_{v}$ is through the logarithm of a concrete $t$-module, whose dimension is huge when $\wt(\fs)$ and $\dep(\fs)$ are large. In this paper, we aim to prove Theorem~\ref{T: IntrodT2} in a more robust way via the logarithmic points of view. Since we have shown in~\cite{CM19b} that the  map $\overline{\mathcal{Z}}\twoheadrightarrow  \overline{\mathcal{Z}}_{v}$ is well-defined and $\ok$-linear,  the statement (1) of Theorem~\ref{T: IntrodT2} is equivalent to the statement (2). It is natural to ask about the kernel of the $\overline{k}$-algebra homomorphism $\overline{\mathcal{Z}}\twoheadrightarrow\overline{\mathcal{Z}}_{v}$ given in Theorem~\ref{T: IntrodT2}. We conjecture that the kernel in question is the principal ideal generated by the single zeta value $\zeta_A(q-1)$ and discuss some applications in Sec.~\ref{Subsec: Questions}.

\subsection{Ideas of the proof} To sketch the key ideas of our proof, we first set up the following notation. Fix any finite place $v$ of $k$. For any index $\fs=(s_{1},\ldots,s_{r})\in \NN^{r}$, let $\Lis_{\fs}(z_{1},\ldots,z_{r})$ be the Carlitz multiple star polylogarithm defined in~\eqref{CMSPL}. We define 
\begin{equation}\label{E: Ds,infty}
 \DD_{\fs,\infty}: =\{ (z_{1}, \dots, z_{r}) \in \CC_{\infty}^{r} | \ |z_{1}|_{\infty} < q^{\frac{s_{1} q}{q-1}} \ \mathrm{and} \ |z_{i}|_{\infty} \leq  q^{\frac{s_{i} q}{q-1}} \ (2 \leq i \leq r) \}\subset \CC_{\infty}^{r}
\end{equation}
and note that by~\cite[Rem.~4.1.3.]{CM19b} $\Lis_{\fs}$ converges $\infty$-adically on $ \DD_{\fs,\infty}$. Concerning the $v$-adic convergence, $\Lis_{\fs}$ converges on 
\begin{equation}\label{E: Dr,v,o}
 \DDConv_{\fs, v} := \{ (z_{1}, \dots, z_{r}) \in \CC_{v}^{r} | \ |z_{1}|_{v} < 1 \ \mathrm{and} \ |z_{i}|_{v} \leq 1 \ (2 \leq i \leq r) \}\subset \CC_{v}^{r},
\end{equation}
but by~\cite{CM19a} it can be extended to be defined on the closed polydisc
\begin{equation}\label{E: Dr,v}
 \DDDef_{\fs, v} := \{ (z_{1}, \dots, z_{r}) \in \CC_{v}^{r} | \  |z_{i}|_{v} \leq 1 \ (1 \leq i \leq r) \}\subset \CC_{v}^{r}.
 \end{equation}
 
We then define the following common sets of algebraic points
 \begin{equation}\label{E: DConv s ok} 
  \DDConv_{\fs, \ok} := \DD_{\fs, \infty} \cap \ok^{\dep \fs} \cap \DDConv_{\fs, v} \subset \ok^{\dep \fs},
 \end{equation} 
 at which $\Lis_{\fs}$ converges both $\infty$-adically and $v$-adically,
 and
 \begin{equation}\label{E: DDef s ok}
  \DDDef_{\fs, \ok} := \DD_{\fs, \infty} \cap \ok^{\dep \fs} \cap \DDDef_{\fs, v} \subset \ok^{\dep \fs},
  \end{equation}
  at which $\Lis_{\fs}$ is defined both $\infty$-adically and $v$-adically.

\begin{definition}\label{Def: Intro} We define the following $\ok$-vector spaces.
\begin{enumerate}
\item[$\bullet$] $\ocLConv_{\infty} :=$ the $\ok$-vector space spanned by $1$ and all $\Lis_{\fs}(\bu)$ for $\fs\in \cup_{r>0}\NN^{r}$ and $\bu \in \DDConv_{\fs, \ok}$.
\item[$\bullet$] $\ocLConv_{v} :=$ the $\ok$-vector space spanned by $1$ and all $\Lis_{\fs}(\bu)_{v}$ for $\fs\in \cup_{r>0}\NN^{r}$ and $\bu \in \DDConv_{\fs, \ok}$.
\item[$\bullet$]  $\ocLDef_{\infty} :=$ the $\ok$-vector space spanned by $1$ and all $\Lis_{\fs}(\bu)$ for $\fs\in \cup_{r>0}\NN^{r}$ and $\bu \in \DDDef_{\fs, \ok}$.
\item[$\bullet$] $\ocLDef_{v} :=$ the $\ok$-vector space spanned by $1$ and all $\Lis_{\fs}(\bu)_{v}$ for $\fs\in \cup_{r>0}\NN^{r}$ and $\bu \in \DDDef_{\fs, \ok}$.
\end{enumerate}
\end{definition}

As $\DDConv_{\fs, \ok} \subset \DDDef_{\fs, \ok}$ we have the natural inclusions:
\[ \ocLConv_{\infty} \subset \ocLDef_{\infty} \hbox{ and } \ocLConv_{v} \subset \ocLDef_{v}. \]
We illustrate our strategy via the following commutative diagram

 \begin{equation}\label{E: DiagramIntrod}
\xymatrix{
\overline{\mathcal{Z}} \ar@{->>}[d]_{\textnormal{\cite{CM19b}}}  \ar@{^{(}->}[rr]^{\textnormal{\cite{C14}}} & & \ocLDef_{\infty} \ar@{->>}[d]^{\textnormal{Thm.~\ref{T: phi-v linear}}}_{\phi_{v}}  \ar@{=}[rr]^{\textnormal{\tiny {Thm.~\ref{theorem:linear-comb}} }} & & \ocLConv_{\infty} \ar@{->>}[d]  \\
\overline{\mathcal{Z}}_{v}  \ar@{^{(}->}[rr]^{\textnormal{Def.}} & & \ocLDef_{v}  & &\ocLConv_{v} \ar@{_{(}->}[ll]_{=}
}
\end{equation}
with the following descriptions:
\begin{enumerate}
\item The inclusion $\overline{\mathcal{Z}} \hookrightarrow \ocLDef_{\infty}$ follows from~\cite[Thm.~5.2.5]{CM19b} (see~\cite[Thm.~5.5.2]{C14} also).
\item The map $\overline{\mathcal{Z}}\twoheadrightarrow \overline{\mathcal{Z}}_{v}$ given by $\zeta_{A}(\fs)\mapsto \zeta_{A}(\fs)_{v}$ is a well-defined $\ok$-linear map by~\cite[Cor.~6.4.3]{CM19b}.
\item We prove in Theorem~\ref{theorem:linear-comb} that the inclusion $\ocLConv_{\infty} \hookrightarrow \ocLDef_{\infty}$ is in fact an equality. 
\item We prove in Theorem~\ref{T: phi-v linear}  that the map $\phi_{v}: \ocLDef_{\infty} \twoheadrightarrow \ocLDef_{v}$ given by $\Lis_{\fs}(\bu)\mapsto \Lis_{\fs}(\bu)_{v}$ is a well-defined $\ok$-linear map.
\item By definition, the restriction of $\phi_{v}$ to $\ocLConv_{\infty}$ is a well-defined $\ok$-linear map onto $\ocLConv_{v}$.
\end{enumerate}

With the above properties established, we mention that since the map $\phi_{v}$ is surjective and $\ocLDef_{\infty} = \ocLConv_{\infty}$, it implies the equality $\ocLDef_{v} = \ocLConv_{v}$. Furthermore, because the CMSPL's converge $\infty$-adically and $v$-adically on $\DDConv_{\fs, \ok}$, the values in $\ocLConv_{\infty}$ and $\ocLConv_{v}$ satisfy the stuffle relations respectively (see~\eqref{infty-adic stuffle} and \eqref{v-adic stuffle}) and hence the map $\phi_{v}$ is a $\ok$-algebra homomorphism. As the restriction of $\phi_{v}$ to $\overline{\mathcal{Z}}$ is given by $\zeta_{A}(\fs)\mapsto \zeta_{A}(\fs)_{v}$ (see~\eqref{E: MZV-formula} and \eqref{E: v-adic MZV-formula}), we derive the desried $\ok$-algebra homomorphism $\overline{\mathcal{Z}}\twoheadrightarrow \overline{\mathcal{Z}}_{v}$.

\subsection{Organization of this paper}
In Sec.~\ref{Sec: preliminaries}, we first review the theory of Anderson on his $t$-modules. We then review CMSPL's and describe stuffle relations. We further recall from~\cite{CM19a, CM19b} how we relate CMSPL's to coordinates of logarithms of certain $t$-modules at specific points as is used to define $v$-adic MZV's.

Sections 3 and 4 are the most technical parts, which are devoted to prove Theorem~\ref{theorem:linear-comb}. Given any  $\Lis_{\fs}(\bu)\in \ocLDef_{\infty}$, ie., $\bu\in \DDDef_{\fs,\ok}$, we mention that $\Lis_{\fs}(\bu)$ is realized as the $\wt (\fs)$-th coordinate of the logarithm of an explicitly constructed $t$-module $G$ defined over $\ok$ at an algebraic point $\bv\in G(\ok)$. To show that $\Lis_{\fs}(\bu)\in \ocLConv_{\infty}$, the key of our strategy is to find a suitable algebraic point $\bv'\in G(\ok)$, at which the logarithm $\log_{G}$ converges both $\infty$-adically and $v$-adically and we use techniques of division points to do the trick. From the functional equation of $\log_{G}$, the $\wt(\fs)$-th coordinate of $\log_{G}(\bv')$ is related to the value $\Lis_{\fs}(\bu)$. 

Moreover, we have to control the $v$-adic size of the point $\bv'$ as is needed when defining the $v$-adic MZV $\zeta_{A}(\fs)_{v}$. However, Papanikolas' computation~\cite{Pp} concerning the leading coefficient matrices of the $t^{m}$-action of the $s$-th tensor power of the Carlitz module $\bC^{\otimes s}$ enables us to ensure that the point $\bv'$ satisfies the desired properties. The crucial result is stated as Theorem~\ref{Theorem1}. 

Section~\ref{Sec: Key identity} is devoted to establish a kind of algebraic functional equations for certain coordinate of the logarithm of certain explicit $t$-module at any algebraic point whenever it is defined. The primary result is given as Theorem~\ref{T: Chen}, which is mainly used to prove the identity $\ocLDef_{\infty}=\ocLConv_{\infty}$. In the final section, we use Yu's sub-$t$-module theorem~\cite{Yu97} to prove in Theorem~\ref{T: phi-v linear}  that the map $\phi_{v}$ is a well-defined $\ok$-linear map, and then give a proof for Theorem~\ref{T: IntrodT2}.

\section{Preliminaries}\label{Sec: preliminaries}
\subsection{Notations}

\begin{tablesymb}\label{tableofsymbols} We use the following symbols throughout this paper. \\
${}$\\
\begin{xtabular}{l r l l}
$\NN$ & $=$ & the set of positive integers.\\
$q$ & $=$ &  a power of a prime number $p$.\\
$\F_q$& $=$ & a finite field of $q$ elements.\\
$A$ &$=$ & $\F_q[\theta]$, the polynomial ring in the variable $\theta$ over $\FF_{q}$.\\
$k$ &$=$ & $\F_q(\theta)$, the quotient field of $A$.\\
$\lvert\cdot\rvert_{\infty}$ &$=$ & the normalized absolute value on $k$ for which $|\theta|_{\infty} = q$.\\
$k_\infty$ &$=$ &$\laurent{\F_q}{1/\theta}$,  the completion of $k$ at the infinite place.\\
$\C_\infty$ &$=$ & $\widehat{\overline{k_{\infty}}}$, the $\infty$-adic completion of an algebraic closure of $k_\infty$.\\
$v$ & $=$ & a monic irreducible polynomial in $A$. \\
$\lvert\cdot\rvert_{v}$ &$=$ & the normalized absolute value on $k$ for which $|v|_{v} = q_{v}^{-1}$, where $q_{v} := q^{\deg_{\theta} v}$.\\
$k_{v}$ &$=$ & the completion of $k$ at $v$.\\
$\C_{v}$ &$=$ & $\widehat{\overline{k_{v}}}$, the $v$-adic completion of an algebraic closure of $k_{v}$.\\
$\ok$ & $=$ & an algebraic closure of $k$ with fixed embeddings into $\CC_{\infty}$ and $\CC_{v}$ over $k$.\\
$\tilde{\Lambda}$ &$=$ &$(\lambda_{r},\ldots,\lambda_{1})$ for any $r$-tuple $\Lambda=(\lambda_{1},\ldots,\lambda_{r})$ of symbols.\\
$\| M \|_{w}$ & $=$ & $\max_{i, j} \{ | M_{ij} |_{w} \}$ for $M = (M_{ij}) \in \Mat_{\ell \times m}(\CC_{w})$ where $w=\infty$ or $w=v$. \\
$\wt(\fs)$&= & $\sum_{i=1}^{r}s_{i}$ for an index $\fs=(s_{1},\ldots,s_{r})\in \NN^{r}$.\\
$\dep(\fs)$&= & $r$ for an index $\fs=(s_{1},\ldots,s_{r})\in \NN^{r}$.\\
\end{xtabular}
\end{tablesymb}

\subsection{$t$-modules associated to CMSPL's} 

\subsubsection{Review of Anderson's theory on $t$-modules}
For any $\FF_{q}$-algebra $R$, any matrix $M = (M_{ij}) \in \Mat_{\ell \times m}(R)$ and any non-negative integer $n$, we define the $n$-th fold Frobenius twist by $M^{(n)} := (M_{ij}^{q^{n}})$. We then define the non-commutative ring $\Mat_{d}(R)[\tau]$, whose elements are of the form $\sum_{i\geq 0} a_{i} \tau^{i}$ with $a_{i} \in \Mat_{d}(R)$ and $a_{i} = 0$ for $i \gg 0$, and whose  multiplication law is given by
\[ (\sum_{i\geq 0}a_{i}\tau^{i})(\sum_{j\geq 0}b_{j}\tau^{j}) =\sum_{i}\sum_{j} a_{i} b_{j}^{(i)} \tau^{i+j}. \]
We put $R[\tau] := \Mat_{1}(R)[\tau]$ then we have natural identifications
\[
\Mat_{d}(R[\tau]) \simeq \Mat_{d}(R)[\tau] \simeq \End_{\Fq}(\GG_{a/R}^{d}),
\]
where $\End_{\Fq}(\GG_{a/R}^{d})$ is the ring of $\Fq$-linear endomorphisms over $R$ of the $d$-dimensional additive group scheme $\GG_{a/R}^{d}$ and $\tau$ corresponds to the Frobenius operator $(x_{1}, \dots, x_{d})^{\tr} \mapsto (x_{1}^{q}, \dots, x_{d}^{q})^{\tr}$.

Let $t$ be a new variable.  Given an $A$-subalgebra $R\subset \ok$ and a positive integer $d$, a $t$-module of dimension $d$ defined over $R$ is an $\FF_{q}$-linear ring homomorphism
\[ [-]:\FF_{q}[t] \rightarrow \Mat_{d}(R[\tau]) \]
so that $\partial [t]-\theta \cdot I_{d}$ is a nilpotent matrix. Here, for $a \in \Fq[t]$ we define $\partial [a]:=a_{0}$ whenever $[a]=\sum_{i=0}^{m} a_{i} \tau^{i}$ for $a_{i}\in \Mat_{d}(R)$. We denote by $G=({\GG_{a}^{d}}_{/R},[-])$, whose underlying space is the group scheme $\GG_{a}^{d}$ over $R$ and whose $\FF_{q}[t]$-module structure is through the $\FF_{q}$-linear ring homomorphism $[-]$, ie, for any $R$-algebra $R'$, the $\FF_{q}[t]$-module structure on ${\GG_{a}^{d}}_{/R}(R') = (R')^{d}$ is given by
\[ a\cdot \bx:=[a](\bx) \]
for $a\in \FF_{q}[t]$ and $\bx\in (R')^{d}$.

Fix a $d$-dimensional $t$-module $G$ defined over $R$ as above, and let $K$ be the fraction field of $R$. Anderson~\cite{A86} showed that there is a $d$-variable $\FF_{q}$-linear formal power series
\[\exp_{G}\in \power{K}{z_{1},\ldots,z_{d}}^{d} \] 
satisfying that $\exp_{G}(\bz)\equiv \bz \pmod {\deg q}$ for $\bz=(z_{1},\ldots,z_{d})^{\tr}$, and as formal power series identity we have

\begin{equation}\label{E: FuncEq of exp}
\exp_{G} \circ \partial [a] = [a] \circ \exp_{G}
\end{equation}for all $a\in \FF_{q}[t]$. The power series $\exp_{G}$ is called the {\it{exponential map}} of $G$, and it is shown by Anderson that as $\infty$-adic convergence, it is an entire function from $\Lie G(\CC_{\infty})=\CC_{\infty}^{d}$ to $G(\CC_{\infty})=\CC_{\infty}^{d}$. The formal inverse of $\exp_{G}$ is denoted by $\log_{G}$ and is called the logarithm map of $G$. So as formal power series one has the following propertities:
\begin{itemize}
\item[$\bullet$] $\log_{G}(\bz) \equiv \bz \pmod {\deg q}$.
\item[$\bullet$] $\log_{G}\circ [a]=\partial [a]\circ \log_{G}$ for $a\in \FF_{q}[t]$.
\end{itemize}

From the point of view of transcendence theory~\cite{Yu91, Yu97}, the values of logarithms of $t$-modules at algebraic points provide rich resources of interesting transcendental numbers. In this paper, we deepen this logarithmic perspective and such logarithmic interpretations for our special values studied here provide the key approaches to prove Theorem~\ref{T: IntrodT2}.

\subsubsection{CMSPL's and stuffle relations}
Let $L_{0} := 1$ and $L_{i} := \prod_{j = 1}^{i} (\theta - \theta^{q^{j}})$ for $i \geq 1$. For any index $\fs = (s_{1}, \dots, s_{r})\in \NN^{r}$, we define the $\fs$-th Carlitz multiple polylogarithm (CMPL) as follows (see~\cite{C14}):
\begin{equation*}
\Li_\fs(z_1,\dots,z_r):=\underset{i_1>\cdots>i_r\geq 0}{\sum}\frac{z_1^{q^{i_1}}\dots z_r^{q^{i_r}}}{L_{i_1}^{s_1}\cdots L_{i_r}^{s_r}}\in \power{k}{ z_1,\cdots,z_r}.
\end{equation*}
We also define the $\fs$-th Carlitz multiple star polylogarithm (CMSPL) as follows (see~\cite{CM19a}):
\begin{equation} \label{CMSPL}
\Li_\fs^\star(z_1,\dots,z_r):=\underset{i_1\geq\cdots\geq i_r\geq 0}{\sum}\frac{z_1^{q^{i_1}}\cdots z_r^{q^{i_r}}}{L_{i_1}^{s_1}\cdots L_{i_r}^{s_r}}\in \power{k} {z_1,\cdots,z_r}.
\end{equation}
We denote by $\Li_\fs(z_1,\cdots,z_r)_v$ and $\Li_\fs^\star(z_1,\cdots,z_r)_v$ when we consider the $v$-adic convergence of those two infinite series.

In what follows, we describe the stuffle relations arising from CMSPL's. Let
\[
X := \{ (s, u) | s \in \NN, u \in \ok, |u|_{\infty} \leq q^{\frac{sq}{q-1}}, |u|_{v} \leq 1 \} \subset \NN \times \ok
\]
and
\[
X^{0} := \{ (s, u) | s \in \NN, u \in \ok, |u|_{\infty} < q^{\frac{sq}{q-1}}, |u|_{v} < 1 \} \subset X.
\]
Let
\[
\hcmspl{1} := \ok \langle z_{s, u} | (s, u) \in X \rangle
\]
be the non-commutative polynomial algebra over $\ok$ generated by the variables $\left\{z_{s, u}\right\}_{(s, u) \in X}$ and
\[
\hcmspl{0} := \ok \oplus \left( \bigoplus_{(s, u) \in X^{0}} z_{s, u} \hcmspl{1} \right) \subset \hcmspl{1}.
\] For any $\fs = (s_{1}, \dots, s_{r}) \in \NN^{r}$ and $\bu = (u_{1}, \dots, u_{r}) \in \ok^{r}$ with $(s_{i},u_{i})\in X$ for $i=1,\ldots,r$, we shall call $z_{s_{1},r_{1}}\cdots z_{s_{r},u_{r}}$ the monomial associated to the pair $(\fs,\bu)$, and vice versa.

We define the $\ok$-bilinear stuffle product $\bst$ on $\hcmspl{1}$ by
\[
1 \bst w = w \bst 1 = w
\]
and
\[
z_{s, u} w \bst z_{s', u'} w' = z_{s, u} (w \bst z_{s', u'} w') + z_{s', u'} (z_{s, u} w \bst w') - z_{s + s', u u'} (w \bst w')
\]
for each $(s, u), (s', u') \in X$ and $w, w' \in \hcmspl{1}$ (cf.~\cite[Sec.~2.1]{IKOO11}). Clearly, $\hcmspl{0}$ is closed under $\bst$. 

Now we define $\cLis(-)$ and $\cLis(-)_{v}$ to be the $\ok$-linear maps given by $\cLis(1) := 1$,  $\cLis(1)_{v} := 1$, and
\[
\cLis(-):=\left(z_{s_{1}, u_{1}} \cdots z_{s_{r}, u_{r}} \mapsto \Lis_{(s_{1}, \dots, s_{r})} (u_{1}, \dots, u_{r})
 \right) \colon \hcmspl{0} \twoheadrightarrow \ocLConv_{\infty} \subset \CC_{\infty},\]
and
\[
\cLis(-)_{v}:=\left( z_{s_{1}, u_{1}} \cdots z_{s_{r}, u_{r}} \mapsto \Lis_{(s_{1}, \dots, s_{r})} (u_{1}, \dots, u_{r})_{v} \right): \hcmspl{0} \twoheadrightarrow \ocLConv_{v} \subset \CC_{v}.
\]

The following describes the stuffle relations for the convergent values of CMSPL's, and it may be well-known for experts but to be self-contained we provide the detailed but short arguments here.
\begin{proposition} \label{stuffle-product-formula}
The $\ok$-linear maps $\cLis(-)$ and $\cLis(-)_{v}$ are multiplicative in the sense that
\[
\cLis(w \bst w') = \cLis(w)\cdot \cLis(w') \ \mathrm{and} \ \cLis(w \bst w')_{v} = \cLis(w)_{v}\cdot \cLis(w')_{v}
\]
for each $w, w' \in \hcmspl{0}$. In particular, $\ocLConv_{\infty}$ and $\ocLConv_{v}$ form $\ok$-algebras and their generators satisfy the same stuffle relations in the sense that for monomials $w,w'\in \hcmspl{0}$ with expression
\begin{equation}\label{E: stuffle multiplication} 
w\bst w'=\sum_{i} \alpha_{i} w_{i},\ \alpha_{i}\in \FF_{p} ,
\end{equation}
we have
\begin{equation}\label{infty-adic stuffle}
\cLis(w)\cdot \cLis(w')=\cLis(w \bst w') =\sum_{i}\alpha_{i}\cLis(w_{i}) \hbox{ in }\CC_{\infty}
\end{equation}
and
\begin{equation}\label{v-adic stuffle}
\cLis(w)_{v} \cdot\cLis(w')_{v}=\cLis(w \bst w')_{v} =\sum_{i}\alpha_{i}\cLis(w_{i})_{v} \hbox{ in }\CC_{v}.
\end{equation}

\end{proposition}
\begin{proof}
For each non-negative integer $n$, we define the (truncated) $\ok$-linear map $\cLisleq{n}$ given by $\cLisleq{n}(1) := 1$ and
\[
\cLisleq{n}:=\left( z_{s_{1},u_{1}} \cdots z_{s_{r}, u_{r}} \mapsto \sum_{n \geq i_{1} \geq \cdots \geq i_{r} \geq 0} \dfrac{u_{1}^{q^{i_{1}}} \cdots u_{r}^{q^{i_{r}}}}{L_{i_{1}}^{s_{1}} \cdots L_{i_{r}}^{s_{r}}} \right) \colon \hcmspl{0} \to \ok.
\]
Since ${\displaystyle \lim_{n \to \infty}} \cLisleq{n}(w) = \cLis(w)$ in $\CC_{\infty}$ and ${\displaystyle \lim_{n \to \infty}} \cLisleq{n}(w) = \cLis(w)_{v}$ in $\CC_{v}$, it suffices to show that
\[
\cLisleq{n}(w \bst w') = \cLisleq{n}(w)\cdot \cLisleq{n}(w')
\]
for all $w, w' \in \hcmspl{0}$ and $n \in \ZZ_{\geq 0}$.

Because of linearity, we may assume that $w$ and $w'$ are monomials. We prove the desired claim by induction on the sum of total degrees of $w$ and $w'$. If $w = 1$ or $w' = 1$, then the equality is clearly valid. Let $w\neq 1, w' \neq 1$ and suppose that the equality holds for all $n$ and for monomials whose total degree is less than $\deg(w) + \deg(w')$. If we write $w = z_{s, u} w_{0}$ and $w' = z_{s', u'} w'_{0}$, then we have
\begin{align*}
& \cLisleq{n}(z_{s, u} w_{0}) \cdot \cLisleq{n}(z_{s', u'} w'_{0}) \\
& = \sum_{n \geq i \geq 0} \dfrac{u^{q^{i}}}{L_{i}^{s}} \cLisleq{i}(w_{0}) \cdot \cLisleq{i}(z_{s', u'} w'_{0})
+ \sum_{n \geq i \geq 0} \dfrac{(u')^{q^{i}}}{L_{i}^{s'}} \cLisleq{i}(z_{s, u} w_{0}) \cdot \cLisleq{i}(w'_{0}) \\
& -\sum_{n \geq i \geq 0} \dfrac{(u u')^{q^{i}}}{L_{i}^{s + s'}} \cLisleq{i}(w_{0})\cdot \cLisleq{i}(w'_{0}) \\
& = \sum_{n \geq i \geq 0} \dfrac{u^{q^{i}}}{L_{i}^{s}} \cLisleq{i}(w_{0} \bst z_{s', u'} w'_{0})
+ \sum_{n \geq i \geq 0} \dfrac{(u')^{q^{i}}}{L_{i}^{s'}} \cLisleq{i}(z_{s, u} w \bst w'_{0}) \\
& - \sum_{n \geq i \geq 0} \dfrac{(u u')^{q^{i}}}{L_{i}^{s + s'}} \cLisleq{i}(w_{0} \bst w'_{0}) \\
& = \cLisleq{n}(z_{s, u} (w_{0} \bst z_{s', u'} w'_{0})) + \cLisleq{n}(z_{s', u'} (z_{s, u} w_{0} \bst w'_{0})) - \cLisleq{n}(z_{s + s', u u'} (w_{0} \bst w'_{0})) \\
& = \cLisleq{n}(z_{s, u} w_{0} \bst z_{s', u'} w'_{0}),
\end{align*}
where the second identity comes from the induction hypothesis and the last two identities follow from definitions. 

\end{proof}

\subsubsection{The construction of $G_{\fs,\bu}$}
	Throughout this section, we fix $\fs = (s_{1}, \dots, s_{r})\in \NN^{r}$ and $\bu = (u_{1}, \dots, u_{r}) \in \ok^{r}$.
For $1 \leq i \leq r$,
we set 
\begin{equation}\label{E: d i}
d_{i} := s_{i} + \cdots + s_{r}
\end{equation}
and 
\begin{equation}\label{E: d}
d := d_{1} + \cdots + d_{r}.
\end{equation}
Let $B$ be a $d \times d$-matrix of the form

\[
\left( \begin{array}{c|c|c}
B[11] & \cdots & B[1r] \\ \hline
\vdots & & \vdots \\ \hline
B[r1] & \cdots & B[rr]
\end{array} \right)
\]
where $B[\ell m]$ is a $d_{\ell} \times d_{m}$-matrix for each $\ell$ and $m$.
We call $B[\ell m]$ the $(\ell, m)$-th block sub-matrix of $B$.

For $1 \leq \ell \leq m \leq r$, we set

\[
N_{\ell} := \left(
\begin{array}{ccccc}
0 & 1 & 0 & \cdots & 0 \\
& 0 & 1 & \ddots & \vdots \\
& & \ddots & \ddots & 0 \\
& & & \ddots & 1 \\
& & & & 0
\end{array}
\right)
\in \Mat_{d_{\ell}}(\ok),
\]

\[
N := \left(
\begin{array}{cccc}
N_{1} & & & \\
& N_{2} & & \\
& & \ddots & \\
& & & N_{r}
\end{array}
\right)
\in \Mat_{d}(\ok),
\]

\[
E[\ell m] := \left(
\begin{array}{cccc}
0 & \cdots & \cdots & 0 \\
\vdots & \ddots & & \vdots \\
0 & & \ddots & \vdots \\
1 & 0 & \cdots & 0
\end{array}
\right)
\in \Mat_{d_{\ell} \times d_{m}}(\ok) \ \ \ (\mathrm{if} \ \ell = m),
\]

\[
E[\ell m] := \left(
\begin{array}{cccc}
0 & \cdots & \cdots & 0 \\
\vdots & \ddots & & \vdots \\
0 & & \ddots & \vdots \\
(-1)^{m-\ell} \prod_{e=\ell}^{m-1} u_{e} & 0 & \cdots & 0
\end{array}
\right)
\in \Mat_{d_{\ell} \times d_{m}}(\ok) \ \ \ (\mathrm{if} \ \ell < m),
\]

\[
E := \left(
\begin{array}{cccc}
E[11] & E[12] & \cdots & E[1r] \\
& E[22] & \ddots & \vdots \\
& & \ddots & E[r-1,r] \\
& & & E[rr]
\end{array}
\right)
\in \Mat_{d}(\ok).
\]

Also, we define the $t$-module $G_{\fs, \bu} := (\GG_{a}^{d}, [-])$ by
\begin{equation}\label{E:Explicit t-moduleCMPL}
  [t] = \theta I_{d} + N + E \tau
  \in \Mat_{d}(\ok[\tau]).
\end{equation}
Note that $G_{\fs,\bu}$ depends  only on $u_{1},\ldots,u_{r-1}$.
Finally, we define
\begin{equation}\label{E:v_s,u}
\bv_{\fs, \bu} :=
\begin{array}{rcll}
\ldelim( {15}{4pt}[] & 0 & \rdelim) {15}{4pt}[] & \rdelim\}{4}{10pt}[$d_{1}$] \\
& \vdots & & \\
& 0 & & \\
& (-1)^{r-1} u_{1} \cdots u_{r} & & \\
& 0 & & \rdelim\}{4}{10pt}[$d_{2}$] \\
& \vdots & & \\
& 0 & & \\
& (-1)^{r-2} u_{2} \cdots u_{r} & & \\
& \vdots & & \vdots \\
& 0 & & \rdelim\}{4}{10pt}[$d_{r}$] \\
& \vdots & & \\
& 0 & & \\
& u_{r} & & \\[10pt]
\end{array} \in G_{\fs,\bu}(\ok).
\end{equation}

\subsection{$v$-adic MZV's} To review the definition of $v$-adic MZV's, we need to recall how we extend the defining domain of the $v$-adic CMSPL $\Lis_{\fs}$ to $\DDDef_{\fs,v}$. Fix an $r$-tuple $\fs=(s_{1},\ldots,s_{r}) \in \NN^{r}$ and note that  $\Lis_{\fs}$ converges $v$-adically on $\DDConv_{\fs, v}$ given in \eqref{E: Dr,v,o} (see~\cite[Sec.~2.2]{CM19a}).
The following result gives logarithmic interpretation for CMSPL's at algebraic points.
\begin{theorem}[\textnormal{\cite[Thm.~3.3.3]{CM19a} and \cite[Thm.~4.2.3]{CM19b}}]\label{T: CMSPL as log} Fixing any $r$-tuples $\fs=(s_{1},\ldots,s_{r})\in \NN^{r}$ and $\bu=(u_{1},\ldots,u_{r})\in \ok^{r}$, let  $G_{\fs,\bu}$ and $\bv_{\fs,\bu}$ be defined as above.

\begin{enumerate}
\item If $\tilde{\bu} = (u_{r},\ldots,u_{1}) \in \ok^{r}\cap \DD_{\tilde{\fs},\infty}\subset \CC_{\infty}^{r}$ defined in \eqref{E: Ds,infty} and $\bx = (x_{i}) \in G_{\fs, \bu}(\CC_{\infty})$ with $|x_{d_{1} + \cdots d_{m-1} + j}|_{\infty} < q^{-(d_{m} - j) + \frac{d_{m} q}{q - 1}}$ for each $1 \leq m \leq r$ and $1 \leq j \leq d_{m}$, then $\log_{G_{\fs, \bu}}(\bx)$ converges $\infty$-adically. In particular, $\log_{G_{\fs,\bu}}\left( \bv_{\fs,\bu} \right)$ converges $\infty$-adically. Moreover, its $\wt(\fs)$-th coordinate is equal to \[ (-1)^{r-1} \cdot \Lis_{\tilde{\fs}}(\tilde{\bu}) = (-1)^{r-1}\cdot \Lis_{(s_{r},\ldots,s_{1})}(u_{r},\ldots,u_{1}).\]

\item If $\tilde{\bu}=(u_{r},\ldots,u_{1}) \in \ok^{r}\cap \DDConv_{\tilde{\fs}, v} \subset \CC_{v}^{r}$ defined in \eqref{E: Dr,v,o} and $\bx \in G_{\fs, \bu}(\CC_{v})$ with $\| \bx \|_{v} < 1$, then $\log_{G_{\fs, \bu}}(\bx)_{v}$ converges $v$-adically. In particular, $\log_{G_{\fs,\bu}}\left( \bv_{\fs,\bu} \right)_{v}$ converges $v$-adically. Moreover, its $\wt(\fs)$-th coordinate is equal to 
\[ (-1)^{r-1} \cdot \Lis_{\tilde{\fs}}(\tilde{\bu})_{v} = (-1)^{r-1}\cdot \Lis_{(s_{r},\ldots,s_{1})}(u_{r},\ldots,u_{1})_{v}.\]
\end{enumerate}
\end{theorem}

\begin{remark}
In fact, all coordinates of $\log_{G_{\fs,\bu}}\left( \bv_{\fs,\bu} \right)$ can be written explicitly in~\cite{CGM19}, and the tractable coordinates (see Definition \ref{def tractable}) of $\log_{G_{\fs,\bu}}\left( \bv_{\fs,\bu} \right)_{v}$ are given explicitly in~\cite{CM19a}.
\end{remark}

We put $v(t):=v|_{\theta=t}\in \FF_{q}[t]$. Define the local ring $\mathcal{O}_{\CC_{v}} := \{ \alpha \in \CC_{v} ; |\alpha|_{v} \leq 1 \}$ and denote by $\fm_{v}$ the maximal ideal of $\mathcal{O}_{\CC_{v}}$.  
The purpose of constructing $G_{\fs, \bu}$ for given $\fs$ and $\bu$ is for the purpose
of connecting $\Lis_{\tilde{\fs}}(\tilde{\bu})$ as well as $\Lis_{\tilde{\fs}}(\tilde{\bu})_{v}$ with a coordinate logarithm of the special algebraic point $\bv_{\fs, \bu}$, where the depth one case was established in~\cite{AT90}. The following provides an approach to extend the $v$-adic defining domains of CMSPL's.

\begin{proposition}[\textnormal{\cite[Prop.~4.1.1]{CM19a}}]\label{proposition-log-converge}
Let $\fs=(s_{1},\ldots,s_{r})\in \NN^{r}$ and $\bu:=(u_{1}, \dots, u_{r}) \in \ok^{r} \cap \DDDef_{\fs, v}$ defined in \eqref{E: Dr,v}. Let $G_{\fs,\bu}$ be the $t$-module defined in \eqref{E:Explicit t-moduleCMPL} and $\bv_{\fs, \bu} \in G_{\fs,\bu}(\ok)$ be defined in \eqref{E:v_s,u}. Let $\ell \geq 1$ be an integer such that each image of $u_{i}$ in $\mathcal{O}_{\CC_{v}} / \fm_{v} \cong \overline{\FF_{q_{v}}}$ is contained in $\FF_{q_{v}^{\ell}}$. Let $d_{1},\ldots,d_{r}$ be defined in \eqref{E: d i}. Then
\[ \bigl|\!\bigl| [v(t)^{d_{1} \ell} - 1] [v(t)^{d_{2} \ell} - 1] \cdots [v(t)^{d_{r} \ell} - 1] \bv_{\fs,\bu} \bigr|\!\bigr|_{v} < 1. \]
In particular,
\[ \log_{G_{\fs,\bu}} \left([v(t)^{d_{1} \ell} - 1] [v(t)^{d_{2} \ell} - 1] \cdots [v(t)^{d_{r} \ell} - 1] \bv_{\fs,\bu} \right)_{v} \]
converges in $\Lie G_{\fs,\bu}(\CC_{v})$.
\end{proposition}

For any $\bu \in \ok^{r} \cap \DDDef_{\fs, v}$ (this is equivalent to $\tilde{\bu} \in \ok^{r} \cap \DDDef_{\tilde{\fs}, v}$), let $a \in \FF_{q}[t]$ be nonzero such that $\bigl| \! \bigl| [a] \bv_{\fs, \bu} \bigr| \! \bigr|_{v} < 1$. The $v$-adic CMSPL $\Lis_{\tilde{\fs}}$ at $\tilde{\bu}$ is defined to be

\[\Li_{(s_{r}, \dots, s_{1})}^{\star}(u_{r}, \dots, u_{1})_{v}:=\dfrac{(-1)^{r-1}}{a(\theta)} \times \left( \textrm{the} \ \wt(\fs) \textrm{-th  coordinate of} \ \log_{G_{\fs,\bu}} \left( [a] \bv_{\fs,\bu} \right)_{v} \right).
\] It is shown in~\cite{CM19a} that the $v$-adic value $\Lis_{\tilde{\fs}}(\tilde{\bu})$ for $\bu \in \ok^{r} \cap \DDDef_{\fs, v}$ is independent of the choices of $a(t)$ and the existence of such $a(t) \in \Fq[t]$ is guaranteed by Proposition \ref{proposition-log-converge}.

We now recall the  formula expressing MZV's as linear combinations of CMSPL's at algebraic points from~\cite[Thm.~5.5.2]{C14} or \cite[Thm.~5.2.5]{CM19b}. For any index $\fs\in \NN^{r}$, there are some explicit tuples $\fs_{\ell}\in \NN^{\dep(\fs_{\ell})}$ with $\dep(\fs_{\ell})\leq \dep(\fs)$ and $\wt(\fs_{\ell})=\wt(\fs)$, some explicit elements $b_{\ell}\in k$ and some explicit integral points
\[ \bu_{\ell}\in A^{\dep(\fs_{\ell})}\cap  \DD_{\fs_{\ell}, \infty}\] so that

\begin{equation}\label{E: MZV-formula}
\zeta_{A}(\fs)= \sum_{\ell}b_{\ell}\cdot (-1)^{\tiny{\rm{dep}(\fs_{\ell})}-1} \Lis_{\fs_{\ell}}(\bu_{\ell}) \in k_{\infty}.
\end{equation}So based on Proposition~\ref{proposition-log-converge} $\Lis_{\fs_{\ell}}(\bu_{\ell})_{v}$ is defined.  The $v$-adic MZV $\zeta_{A}(\fs)_{v}$ in~\cite{CM19b} is defined by

\begin{equation}\label{E: v-adic MZV-formula}  
\zeta_{A}(\fs)_{v}:= \sum_{\ell}b_{\ell}\cdot (-1)^{\tiny{\rm{dep}(\fs_{\ell})}-1} \Lis_{\fs_{\ell}}(\bu_{\ell})_{v} \in k_{v},
\end{equation} which can be thought of an analogue of Furusho's $p$-adic MZV's.  Such as the $\infty$-adic case, we define the weight and depth of the presentation $\zeta_{A}(\fs)_{v}$ to be $\wt(\fs)$ and $\dep(\fs)$ respectively.  These $\infty$-adic MZV's and $v$-adic MZV's are expressed in terms of CMSPL's at algebraic points and  have nice logarithmic interpretations in~\cite{CM19b}, extending Anderson-Thakur's work~\cite{AT90} from depth one to arbitrary depth.

\section{A trick on division points}

    The main theme in this section is to demonstrate that we are able to find a specific algebraic point at which the logarithm of the $t$-module in question converges both $\infty$-adically and $v$-adically. This trick is achieved using the uniformization of the $t$-module as well as its property of iterated extensions of Carlitz tensor powers.

    \subsection{The crucial result}

    Recall the notation $ \DDDef_{\fs, \ok} = \DD_{\fs, \infty} \cap \ok^{\dep \fs} \cap \DDDef_{\fs, v} \subset \ok^{\dep \fs}$ given in Sec.~\ref{E: DDef s ok}. Note that $\exp_{G}$ is locally isometric for any $t$-module $G$ defined over $\ok$, so we can find a small domain $\mathcal{D}_{G}\subset \Lie G(\CC_{\infty})$ on which $\exp_{G}$ is an isometry (see \cite[Lem.~5.3]{HJ16}). For each $x \in \cO_{\CC_{v}}$, we denote by $\overline{x}$ the image of $x$ in the residue field $\cO_{\CC_{v}} / \fm_{v} \cong \overline{\FF_{q_{v}}}$.
For each $\ell \in \NN$, we define a local ring $A_{(v), \ell}$ by
\[
A_{(v), \ell} := \{ x \in \ok \cap \cO_{\CC_{v}} | \overline{x} \in \FF_{q_{v}^{\ell}} \}.
\]
For each $n \in \NN$ and $\bx = (x_{i}) \in \Mat_{n \times 1}(\cO_{\CC_{v}})$, we define $\overline{\bx} := (\overline{x_{i}}) \in \Mat_{n \times 1}(\overline{\FF_{q_v}})$.

    \begin{theorem}\label{Theorem1}
        Let $\fs=(s_{1},\ldots,s_{r}) \in \NN^{r}$, $\bu \in \ok^{r}$ with $\tilde{\bu} \in \DDDef_{\tilde{\fs}, \ok}$ defined in~\eqref{E: DDef s ok}, and $G_{\fs, \bu}=(\mathbb{G}_a^d, [-])$ be the $t$-module defined in \eqref{E:Explicit t-moduleCMPL}. Given any $\bv\in G_{\fs, \bu}(\ok \cap \cO_{\CC_{v}})$ such that $\log_{G_{\fs, \bu}}(\bv) \in \Lie G_{\fs,\bu}(\CC_{\infty})$ converges $\infty$-adically. Pick a positive integer $\ell$ for which $\bv \in G_{\fs, \bu}(A_{(v), \ell})$ and define
      \[  a(t):=(v(t)^{d_1 \ell}-1)\cdots(v(t)^{d_r \ell}-1).\] 
       For each $n \in \ZZ_{\geq 0}$, put
 \[ Z_{n} := \partial[v(t)^{n}]^{-1} \log_{G_{\fs, \bu}}(\bv) \in \Lie G_{\fs, \bu}(\CC_{\infty}) \ \mathrm{and} \ \bv_{n} := \exp_{G_{\fs,\bu}}\left( Z_{n} \right) \in G_{\fs,\bu}(\CC_{\infty}).\]       
Then the following properties hold.
        \begin{enumerate}
            \item $[v(t)^{n}] \bv_{n} = \bv$ and $\bv_{n} \in G_{\fs, \bu}(\ok)$.
            \item $\bigl| \! \bigl| [a(t)] \bv_{n} \bigr| \! \bigr|_{v} < 1$ if $n$ is divisible by $\lcm(d_{1}, \dots, d_{r})$. In particular, $\log_{G_{\fs, \bu}}([a(t)] \bv_{n})_{v}$ converges $v$-adically (see Theorem \ref{T: CMSPL as log}).
            \item $\lim_{n \to \infty} \bigl| \! \bigl| Z_{n} \bigr| \! \bigr|_{\infty} = 0$. In particular, $\log_{G_{\fs,\bu}}\left( [a(t)] \bv_{n} \right)$, $\partial[a(t)] \log_{G_{\fs,\bu}}(\bv_{n})$ and $Z_{n} = \partial[v(t)^{n}]^{-1} \log_{G_{\fs, \bu}}(\bv)$ converge $\infty$-adically, and they are contained in $\mathcal{D}_{G_{\fs,\bu}}$ for sufficiently large $n$.

        \end{enumerate}
    \end{theorem}

In order to prove the theorem above, we need to establish some lemmas. For any positive integer $s$, we denote by $\bC^{\otimes s}=(\GG_{a}^{s},[-]_{s})$ the $s$-th tensor power of the Carlitz module, which is defined over $A$. Precisely,
\begin{align} \label{def-carlitz}
[t]_{s} := \left(
\begin{array}{cccc}
\theta & 1 & \cdots & 0 \\
& \theta & \ddots & \vdots \\
& & \ddots & 1 \\
& & & \theta
\end{array}
\right)
+  \left(
\begin{array}{cccc}
0 & 0 & \cdots & 0 \\
0 & 0 & \cdots & 0 \\
\vdots & \vdots & & \vdots \\
1 & 0 & \cdots & 0
\end{array}
\right) \tau
\in \Mat_{s}(A[\tau]).\end{align}
By \cite[Cor.~3.5.4]{Pp}  for each $m \in \ZZ_{\geq 0}$ with $m \equiv \ell \pmod s \ (0 < \ell \leq s)$, the leading coefficient matrix of $[t^{m}]_{s}$ is of the form
\[ \left(
\begin{array}{cccccc}
0 & \cdots & \cdots & \cdots & \cdots & 0 \\
\vdots & \ddots & & & & \vdots \\
0 & & \ddots & & & \vdots \\
1 & \ddots & & \ddots & & \vdots \\
* & \ddots & \ddots & & \ddots & \vdots \\
* & * & 1 & 0 & \cdots & 0
\end{array}
\right)
\]
where the $(i, j)$-th component is $1$ (resp.\ $0$) if $i = j + s - \ell$ (resp.\ $i < j + s - \ell$). Furthermore, we have
\[ \deg_{\tau} [t^{m}]_{s} = \left\lceil \dfrac{m}{s} \right\rceil. \] 
So the above  properties directly lead to the following.

\begin{lemma}\label{L: P-equations} Let $s$ be a positive integer and
let $b(t) \in \Fq[t]$ be a monic polynomial such that $\deg_{t} b(t)$ is divisible by $s$. Then there exist polynomials $f_{ij}(X)\in  X A[X]$ with $\deg_{X}f_{ij}(X)\leq q^{(\deg_{t} b(t))/s}$ and $\deg_Xf_{ij}(X)<q^{(\deg_{t} b(t))/s}$ if $i\leq j$ so that 
for any 
$\bx=(x_{1},\ldots,x_{s})^{\tr}\in \bC^{\otimes s}(\ok) = \ok^{s}$ and $(y_{1},\ldots,y_{s})^{\tr}:=[b(t)]_{s} \bx$, we have 
        \begin{equation} \label{eq1}
            \begin{split}
                x_{1}^{q^{(\deg_{t} b(t))/s}}+\sum_{j=1}^{s}f_{1j}(x_j) & = y_1 \\
                x_{2}^{q^{(\deg_{t} b(t))/s}}+\sum_{j=1}^{s}f_{2j}(x_j) & = y_2 \\
                & \vdots\\
                x_{s}^{q^{(\deg_{t} b(t))/s}}+\sum_{j=1}^{s}f_{s j}(x_j) & =y_{s}.
            \end{split}
        \end{equation}
        
\end{lemma}

Then Lemma~\ref{L: P-equations} enables us to ensure the $v$-adic integrality of certain $v(t)^{n}$-division points in $G_{\fs,\bu}$.   

    \begin{lemma}\label{Lemma3}
    	Given $\fs=(s_{1},\ldots,s_{r}) \in \NN^{r}$ and $\bu=(b_{1},\ldots,u_{r}) \in \ok^{r}$, let $G_{\fs, \bu}=(\mathbb{G}_a^d, [-])$ be the $t$-module defined in \eqref{E:Explicit t-moduleCMPL}. Let $n \in \NN$ be a positive integer divisible by $\mathrm{lcm}(d_{1}, \dots, d_{r})$. Fix any algebraic point $\bv\in G_{\fs, \bu}(\ok)$ and suppose that $\bv' \in G_{\fs,\bu}(\ok)$ satisfies $[v(t)^{n}] \bv' = \bv$. If $\bu \in A_{(v), \ell}^{r}$ and $\bv \in G_{\fs, \bu}(A_{(v), \ell})$ for some $\ell \in \NN$, then $\bv'$ also lies in $G_{\fs, \bu}(A_{(v), \ell})$.
    \end{lemma}
    
    \begin{proof}
We write \[
[v(t)^{n}] = \begin{pmatrix} [v(t)^{n}]_{d_{1}} & M_{12} & \cdots & M_{1r} \\ & [v(t)^{n}]_{d_{2}} & \cdots & M_{2r} \\ & & \ddots & \vdots \\ & & & [v(t)^{n}]_{d_{r}} \end{pmatrix}\in \Mat_{d}(\ok[\tau]) ,
\] and note that since $\bu\in A_{(v),\ell}^{r}$, we have $M_{ij} \in \Mat_{d_{i} \times d_{j}}(A_{(v), \ell}[\tau])$ for all $i,j$. Hence we have $[v(t)^{n}]\in \Mat_{d}(A_{(v),\ell}[\tau])$ as $[v(t)^{n}]_{d_{i}}\in \Mat_{d_{i}}(A[\tau])$ for $i=1,\ldots,r$.

We further write
\[
\bv = \left( \begin{array}{c} \bv_{1} \\ \vdots \\ \bv_{r} \end{array} \right) \ \ (\bv_{i} \in \Mat_{d_{i}\times 1}(A_{(v), \ell})) \ \mathrm{and} \ \bv' = \left( \begin{array}{c} \bv'_{1} \\ \vdots \\ \bv'_{r} \end{array} \right) \ \ (\bv'_{i} \in \Mat_{d_{i}\times 1}(\ok)),
\]then the equation $[v(t)^{n}] \bv' = \bv$ is expressed as

\begin{equation}\label{E: Equa v'}
[v(t)^{n}]_{d_{i}} \bv'_{i} = \bv_{i} - \sum_{i < j \leq r} M_{ij} \bv'_{j} \ \ (1 \leq i \leq r).
\end{equation}

By considering \eqref{E: Equa v'} for $i=r$ first,  followed by $i=r-1,\ldots,i=1$ inductively, we are reducing to showing the case when $r=1$, ie., $\fs = (s)$. Moreover, due to the fact $[v(t)^{n}]=[v(t)^{s}]\circ \cdots \circ [v(t)^{s}]$ ($\frac{n}{s}$ times) we may assume $n = s$.

Write $\bv'=(x_{1},\ldots,x_{s})^{\tr}$ and $\bv=(y_{1},\ldots,y_{s})^{\tr}$ and then we have the system of equations~\eqref{eq1} for $b(t) = v(t)^{s}$. Let $1 \leq i \leq s$ be the minimal integer such that $|x_{i}|_{v} = \max_{1 \leq j \leq s} \{ |x_{j}|_{v} \}$. 
We first claim that $\bv'\in \bC^{\otimes s}(\ok\cap \cO_{\CC_{v}})$.
Suppose on the contrary that $|x_{i}|_{v} > 1$. Recall that $q_{v} := q^{\deg_{\theta} v}=q^{\frac{\deg_{t} b(t)}{s}}$. Then we have
\[
|f_{ij}(x_{j})|_{v} \leq \max \{ 1, |x_{j}|_{v} \}^{\deg_{X} f_{ij}(X)} \leq \left\{ \begin{array}{ll} |x_{i}|_{v}^{\deg_{X} f_{ij}(X)} & (j \geq i) \\ \max \{ 1, |x_{j}|_{v} \}^{q_{v}} & (j < i) \end{array} \right. < |x_{i}|_{v}^{q_{v}}
\]
for each $1 \leq j \leq s$. Therefore
\[
1 < |x_{i}|_{v}^{q_{v}} = \left| x_{i}^{q_{v}} + \sum_{j=1}^{s} f_{ij}(x_{j}) \right|_{v} = |y_{i}|_{v} \leq 1
\]
which leads to a contradiction, and hence $\bv' \in \bC^{\otimes s}(\ok \cap \cO_{\CC_{v}})$.

By \cite[Proposition 1.6.1]{AT90}, we have
\[[v(t)^{s}]_{s} \bv' \equiv \bv'^{(\deg_{\theta} v)}  \pmod {\fm_{v}}.
\]
So the relation $[v(t)^{s}]_{s} \bv' \equiv \bv \pmod {\fm_{v}}$ implies
\[
\overline{\bv'}^{(\deg_{\theta} v)} = \overline{\bv} \in \bC^{\otimes s}(\FF_{q_{v}^{\ell}}).
\]
Since $\FF_{q_{v}^{\ell}}$ is perfect, we have $\overline{\bv'} \in \bC^{\otimes s}(\FF_{q_{v}^{\ell}})$, and hence $\bv' \in \bC^{\otimes s}(A_{(v), \ell})$.
    \end{proof}

In what follows, we need the notion of regular $t$-modules introduced by Yu in~\cite[p.~218]{Yu97}.

\begin{definition} Let $G$ be a $t$-module defined over $\ok$. We say that $G$ is {\it regular} if there is a positive integer $\nu$ for which the $a$-torsion submodule of $G(\ok)$ is free of rank $\nu$ over $\FF_{q}[t]/(a)$ for every nonzero polynomial $a\in \FF_{q}[t]$. 
\end{definition}

The following proposition is needed in the following subsection, where we prove Theorem~\ref{Theorem1}. Moreover, it also enables us to apply Yu's sub-$t$-module theorem for the $t$-module $G_{\fs, \bu}$ when proving Theorem~\ref{T: IntrodT2} in Section \ref{final-section}.

\begin{proposition} \label{prop-regular}
For any $\fs \in \NN^{r}$ and $\bu \in \ok^{r}$ with $\tilde{\bu} \in \DD_{\tilde{\fs}, \infty}$ defined in \eqref{E: Ds,infty}, the $t$-module $G_{\fs, \bu}$ is regular.
\end{proposition}

\begin{proof}
The $t$-module $G_{\fs, \bu}$ is uniformizable by \cite[Rem.~3.3.5]{CM19a} and \cite[Rem.~2.5.2]{CM19b}. Then by the same arguments of \cite[Prop.~6.3.2]{CM19b}, the desired result follows.
\end{proof}

    \subsection{Proof of Theorem~\ref{Theorem1}} Now we give a proof of Theorem~\ref{Theorem1}. For each $n \in \ZZ_{\geq 0}$, we have
    \begin{align*}
    [v(t)^{n}] \bv_{n} &:= [v(t)^{n}] \exp_{G_{\fs, \bu}} \left( \partial [v(t)^{n}]^{-1} \log_{G_{\fs, \bu}}(\bv) \right) = \exp_{G_{\fs, \bu}} \left( \partial [v(t)^{n}] \partial [v(t)^{n}]^{-1} \log_{G_{\fs, \bu}}(\bv) \right) \\
    &= \exp_{G_{\fs, \bu}} \left( \log_{G_{\fs, \bu}}(\bv) \right) = \bv. \end{align*}
	
	Denote by $G_{\fs, \bu}[v(t)^{n}]$ the algebraic subgroup of $v(t)^{n}$-torsion points of $G_{\fs,\bu}$. As the algebraic variety defined by $[v(t)^{n}] \bX = \bv$ is a $G_{\fs, \bu}[v(t)^{n}]$-torsor, it is a zero-dimensional variety defined over $\ok$ since $G_{\fs, \bu}$ is a regular $t$-module by~Proposition \ref{prop-regular}. Therefore we have \[ \bv_{n} \in G_{\fs, \bu}(\ok), \]  whence proving the property $(1)$.
	
	To prove the property $(2)$, we mention that by Lemma~\ref{Lemma3}, if $n$ is divisible by $\mathrm{lcm}(d_{1}, \dots, d_{r})$ then we have \[ \bv_{n} \in G_{\fs, \bu}(A_{(v), \ell}). \] Furthermore, from \cite[Prop.~4.1.1]{CM19a} the choice of $a(t)$  implies that \[ \bigl| \! \bigl| [a(t)] \bv_{n} \bigr| \! \bigr|_{v} < 1 \] for each $n \in \NN$ with $\mathrm{lcm}(d_{1}, \dots, d_{r}) | n$.

Now we aim to show that $\bigl| \! \bigl| Z_{n} \bigr| \! \bigr|_{\infty} \to 0$ as $n \to \infty$.
Let $N_{b} := \partial [b(t)] - b(\theta) I_{d}$ for each nonzero polynomial $b(t) \in \Fq[t]$. Then $N_{b}$ is a nilpotent matrix. Recall that $N$ is the nilpotent matrix given in $[t]$ in~\eqref{E:Explicit t-moduleCMPL}. 
Since
\[
N \in \Mat_{d}(\FF_{p}) \ \mathrm{and} \
\partial [t^{i}] = (\theta I_{d} + N)^{i} = \theta^{i} I_{d} + \sum_{0 \leq j < i} \binom{i}{j} \theta^{j} N^{i-j} \ (i \in \ZZ_{\geq 0}),
\]
we have $|\!| N_{b} |\!|_{\infty} < |b(\theta)|_{\infty}$. Therefore we have
\[
\bigl| \! \bigl| \partial[b(t)]^{-1} \bigr| \! \bigr|_{\infty} = \Biggl| \! \Biggl| b(\theta)^{-1} \left( \sum_{\ell = 0}^{d - 1} (- N_{b} / b(\theta))^{\ell} \right) \Biggr| \! \Biggr|_{\infty} = |b(\theta)|_{\infty}^{-1} \to 0 \ \ (\deg_{t} b(t) \to \infty)
\]
and hence
\[
\bigl| \! \bigl| Z_{n} \bigr| \! \bigr|_{\infty} = \bigl| \! \bigl| \partial [v(t)^{n}]^{-1} \log_{G_{\fs, \bu}}(\bv) \bigr| \! \bigr|_{\infty} \to 0 \ \ (n \to \infty).
\]

Finally, since $\exp_{G_{\fs, \bu}}(-)$, $\log_{G_{\fs, \bu}}(-)$, $[a(t)](-)$ and $\partial[a(t)](-)$ are continuous, the second part of the property $(3)$ follows.

    \section{The key identity}\label{Sec: Key identity}

    The main purpose of this section is to show Theorem~\ref{theorem:linear-comb}.

    \subsection{Formula for the weight coordinate}
In this subsection, we aim to give an formula for the $\wt(\fs)$-th coordinate of $\log_{G_{\fs,\bu}}$ at algebraic points in terms of CMSPL's for each $\fs \in \NN^{r}$ and $\bu \in \ok^{r}$ with $\tilde{\bu} \in \DD_{\tilde{\fs}, \infty}$. We first mention that the $\infty$-adic convergence domain of $\log_{G_{\fs,\bu}}$ is given in Theorem~\ref{T: CMSPL as log}. We then recall that the $\wt(\fs)$-th row of the coefficient matrices of $\log_{G_{\fs,\bu}}$ is explicitly given as follows.

    \begin{proposition}[{\cite[Prop.~3.2.1]{CM19a}, \cite[Prop.~3.2.2]{Chen20}}] \label{P: Chen20} Fix any $\fs=(s_{1},\ldots,s_{r})\in \NN^{r}$ and $\bu=(u_{1},\ldots,u_{r})\in \ok^{r}$, let $G_{\fs,\bu}$ be defined in \eqref{E:Explicit t-moduleCMPL}, $d_i$ be given in \eqref{E: d i} for $i=1,\ldots,r$ and $d$ be given in \eqref{E: d}. We 
        put \[\log_{G_{\fs,\bu}}:=\underset{i\geq 0}{\sum}P_i\tau^i, \ P_0:=I_d, \ P_{i} \in \Mat_{d}(\ok), \] and write the $\wt(\fs)$-th row of $P_{i}$ as 
        \[ \left( y_{1,1}^{<i>},\ldots,y_{1,d_1}^{<i>}, y_{2,1}^{<i>},\ldots,y_{2,d_2}^{<i>},\ldots, y_{r,1}^{<i>},\ldots,y_{r,d_r}^{<i>} \right).  \]Then for each $i \geq 0$, we have
        
        \begin{equation}\label{y^(i)_1,j}
            y_{1,j}^{<i>}=\frac{(\theta-\theta^{q^{i}})^{d_1-j}}{L_i^{d_1}} \hbox{ for }1\leq j\leq d_1,
        \end{equation}
        and  for each $2\leq m\leq r$ and $1\leq j\leq d_{m}$ we have
        \begin{equation}\label{y^(i)_m,j}
            y_{m,j}^{<i>} = (-1)^{m-1}(\theta-\theta^{q^{i}})^{d_m-j}\underset{0\leq i_1\leq\cdots\leq i_{m-1}<i}{\sum} \frac{u_1^{q^{i_{1}}} \cdots u_{m-1}^{q^{i_{m-1}}}}{L_{i_1}^{s_1}\cdots L_{i_{m-1}}^{s_{m-1}}L_i^{d_m}}.
        \end{equation}
    \end{proposition}

    As a consequence of the proposition above, we obtain the following crucial identity.    

    \begin{theorem}[cf.~{\cite[Thm.~3.2.9]{Chen20}}]\label{T: Chen}
        Fix any $\fs=(s_{1}, \dots, s_{r})\in\mathbb{N}^{r}$ and $\bu=(u_{1}, \dots, u_{r}) \in \ok^{r}$ with $\tilde{\bu} \in \DD_{\tilde{\fs}, \infty}$ defined in \eqref{E: Ds,infty}. Let $G_{\fs,\bu}$ be defined in \eqref{E:Explicit t-moduleCMPL} and $d_i$ be given in \eqref{E: d i} for $i=1,\ldots,r$.
        If we set 
        \[\bx:=\left(x_{1,1}, \dots, x_{1, d_{1}},x_{2,1}, \dots, x_{2, d_{2}},\ldots,x_{r,1}, \dots, x_{r, d_{r}} \right)^{\tr}\in G_{\fs, \bu}(\CC_{\infty}), \] and assume \[ |x_{m, j}|_{\infty} < q^{- (d_{m} - j) + \frac{d_{m} q}{q - 1}} \] for each $1 \leq m \leq r$ and $1 \leq j \leq d_{m}$, then
        \[
	|\theta^{\ell} x_{m, j}|_{\infty} < q^{\frac{d_{m} q}{q - 1}} \ \ \ (1 \leq m \leq r, \ 1 \leq j \leq d_{m}, \ 0 \leq \ell \leq d_{m} - j)
	\]
	and
	\[
	|\theta^{\ell} x_{m, j} u_{m-1}|_{\infty} < q^{\frac{d_{m-1} q}{q - 1}} \ \ \ (2 \leq m \leq r, \ 1 \leq j \leq d_{m}, \ 0 \leq \ell \leq d_{m} - j),
	\]
        and the $\wt(\fs)$-th coordinate of $\log_{G_{\fs, \bu}}(\bx)$ is given by 
        \begin{align*} & \sum_{j = 1}^{d_{1}} \sum_{\ell = 0}^{d_{1} - j} (-1)^{\ell} \binom{d_{1} - j}{\ell} \theta^{d_{1} - j - \ell} \Li^{\star}_{d_{1}}(\theta^{\ell} x_{1, j}) \\
        & + \sum_{2 \leq m \leq r} (-1)^{m - 1} \sum_{j = 1}^{d_{m}} \sum_{\ell = 0}^{d_{m} - j} (-1)^{\ell} \binom{d_{m} - j}{\ell} \theta^{d_{m} - j - \ell} \\
        & \times \left\{ \Li^{\star}_{(d_{m}, s_{m-1}, \dots, s_{1})}(\theta^{\ell} x_{m, j}, u_{m-1}, \dots, u_{1}) - \Li^{\star}_{(d_{m-1}, s_{m-2}, \dots, s_{1})}(\theta^{\ell} x_{m, j} u_{m-1}, u_{m-2}, \dots, u_{1}) \right\}, \end{align*}
        where in the case of $m=2$, we denote by
        \[ (d_{m-1}, s_{m-2}, \dots, s_{1}) := (d_{1}) \ \mathrm{and} \ (\theta^{\ell} x_{m, j} u_{m-1}, u_{m-2}, \dots, u_{1}) := (\theta^{\ell} x_{2, j} u_{1}). \]
       
    \end{theorem}
    
\begin{proof}
We use the same notations as in Proposition \ref{P: Chen20}.
First of all, we note that since
\[
|\theta^{\ell} x_{m, j}|_{\infty} < q^{\ell} \cdot q^{- (d_{m} - j) + \frac{d_{m} q}{q - 1}} \leq q^{d_{m} - j} \cdot q^{- (d_{m} - j) + \frac{d_{m} q}{q - 1}} = q^{\frac{d_{m} q}{q - 1}}
\]
for each $1 \leq m \leq r$, $1 \leq j \leq d_{m}$ and $0 \leq \ell \leq d_{m} - j$, and
\[
|\theta^{\ell} x_{m, j} u_{m-1}|_{\infty} < q^{\ell} \cdot q^{- (d_{m} - j) + \frac{d_{m} q}{q - 1}} \cdot q^{\frac{s_{m-1} q}{q - 1}} \leq q^{d_{m} - j} \cdot q^{- (d_{m} - j) + \frac{d_{m} q}{q - 1}} \cdot q^{\frac{s_{m-1} q}{q - 1}} = q^{\frac{d_{m-1} q}{q - 1}},
\]
for each $2 \leq m \leq r$, $1 \leq j \leq d_{m}$ and $0 \leq \ell \leq d_{m} - j$, each CMSPL in Theorem \ref{T: Chen} converges $\infty$-adically.

According to Theorem~\ref{T: CMSPL as log}, $\log_{G_{\fs,\bu}}(\bx)$ converges $\infty$-adically.
Since we write $\log_{G_{\fs,\bu}}=\sum_{i=0}^{\infty}{P_{i}}\tau^{i}$,  the $\wt(\fs)$-th coordinate of $\log_{G_{\fs, \bu}}(\bx)$ is given by
\begin{align*}
\sum_{i = 0}^{\infty} \sum_{m = 1}^{r} \sum_{j = 1}^{d_{m}} y_{m, j}^{<i>} x_{m, j}^{q^{i}} .	\end{align*}
We claim that the series 
\[\sum_{m = 1}^{r} \sum_{j = 1}^{d_{m}} \sum_{i = 0}^{\infty} y_{m, j}^{<i>} x_{m, j}^{q^{i}} \]
converges $\infty$-adically, and so it is equal to the $\wt(\fs)$-th coordinate of $\log_{G_{\fs, \bu}}(\bx)$.	

To prove the claim above, we compute $\sum_{i = 0}^{\infty} y_{m, j}^{<i>} x_{m, j}^{q^{i}}$ for each $1 \leq m \leq r$ and $1 \leq j \leq d_{m}$.
When $m = 1$, we have
\begin{align*}
\sum_{i = 0}^{n} y_{1, j}^{<i>} x_{1, j}^{q^{i}}
& = \sum_{i = 0}^{n} (\theta-\theta^{q^{i}})^{d_{1} - j} \dfrac{x_{1, j}^{q^{i}}}{L_{i}^{d_{1}}}
= \sum_{i = 0}^{n} \sum_{\ell = 0}^{d_{1} - j} \binom{d_{1} - j}{\ell} \theta^{d_{1} - j - \ell} (- \theta^{q^{i}})^{\ell} \dfrac{x_{1, j}^{q^{i}}}{L_{i}^{d_{1}}} \\
& = \sum_{\ell = 0}^{d_{1} - j} (-1)^{\ell} \binom{d_{1} - j}{\ell} \theta^{d_{1} - j - \ell} \sum_{i = 0}^{n} \dfrac{(\theta^{\ell} x_{1, j})^{q^{i}}}{L_{i}^{d_{1}}} \\
& \to \sum_{\ell = 0}^{d_{1} - j} (-1)^{\ell} \binom{d_{1} - j}{\ell} \theta^{d_{1} - j - \ell} \Li^{\star}_{d_{1}}(\theta^{\ell} x_{1, j})
\end{align*}
as $n \to \infty$ for each $1 \leq j \leq d_{1}$. When $2 \leq m \leq r$, we have
\begin{align*}
& \sum_{i = 0}^{n} y_{m, j}^{<i>} x_{m, j}^{q^{i}}
= \sum_{i = 0}^{n} (-1)^{m - 1}(\theta-\theta^{q^{i}})^{d_{m} - j} \sum_{0 \leq i_{1} \leq \cdots \leq i_{m - 1} < i} \frac{u_{1}^{q^{i_{1}}} \cdots u_{m - 1}^{q^{i_{m - 1}}} x_{m, j}^{q^{i}}}{L_{i_{1}}^{s_{1}} \cdots L_{i_{m - 1}}^{s_{m - 1}} L_{i}^{d_{m}}} \\
& = \sum_{i = 0}^{n} (-1)^{m - 1} \sum_{\ell = 0}^{d_{m} - j} \binom{d_{m} - j}{\ell} \theta^{d_{m} - j - \ell} (- \theta^{q^{i}})^{\ell} \sum_{0 \leq i_{1} \leq \cdots \leq i_{m - 1} < i} \frac{u_{1}^{q^{i_{1}}} \cdots u_{m - 1}^{q^{i_{m - 1}}} x_{m, j}^{q^{i}}}{L_{i_{1}}^{s_{1}} \cdots L_{i_{m - 1}}^{s_{m - 1}} L_{i}^{d_{m}}} \\
& = (-1)^{m - 1} \sum_{\ell = 0}^{d_{m} - j} (-1)^{\ell} \binom{d_{m} - j}{\ell} \theta^{d_{m} - j - \ell} \sum_{0 \leq i_{1} \leq \cdots \leq i_{m - 1} < i_{m} \leq n} \frac{u_{1}^{q^{i_{1}}} \cdots u_{m - 1}^{q^{i_{m - 1}}} (\theta^{\ell} x_{m, j})^{q^{i_{m}}}}{L_{i_{1}}^{s_{1}} \cdots L_{i_{m - 1}}^{s_{m - 1}} L_{i_{m}}^{d_{m}}} \\
& \to (-1)^{m - 1} \sum_{\ell = 0}^{d_{m} - j} (-1)^{\ell} \binom{d_{m} - j}{\ell} \theta^{d_{m} - j - \ell} \\
& \times \left\{ \Li^{\star}_{(d_{m}, s_{m-1}, \dots, s_{1})}(\theta^{\ell} x_{m, j}, u_{m-1}, \dots, u_{1}) - \Li^{\star}_{(d_{m-1}, s_{m-2}, \dots, s_{1})}(\theta^{\ell} x_{m, j} u_{m-1}, u_{m-2}, \dots, u_{1}) \right\}
\end{align*}
as $n \to \infty$ for each $1 \leq j \leq d_{m}$.
Hence we complete the proof of the claim as well as the desired identity.
\end{proof}

	\subsection{Application of Theorem~\ref{T: Chen}}
Consider the following $\ok$-vector spaces
\[ \ocLConv_{\infty} \subset \ocLDef_{\infty}\subset \CC_{\infty} \]
given in Definition~\ref{Def: Intro}. An important application of Theorem~\ref{T: Chen} is the following equality, which is the key for us to prove Theorem~\ref{T: IntrodT2} in the next section.

\begin{theorem}\label{theorem:linear-comb} Let notation be given in Definition~\ref{Def: Intro}. Then we have the following equality
\[ \ocLConv_{\infty} = \ocLDef_{\infty}.\]
\end{theorem}

\begin{proof}
Suppose ${\bf{n}}=(n_{1},\ldots,n_{r})\in \NN^{r}$ and ${\bf{w}}=(w_{1},\ldots,w_{r})\in \DDDef_{\bf{n},\ok}$ (so $\Lis_{\bf{n}}({\bf{w}})\in \ocLDef_{\infty}$), our goal is to show that $\Lis_{\bf{n}}({\bf{w}})\in \ocLConv_{\infty}$.

Put
\[ \fs:=\widetilde{\bf{n}}=(n_{r},\ldots,n_{1}), \ \bu:=\widetilde{\bf{w}}=(w_{r},\ldots,w_{1}) \]
and write $\fs=(s_{1},\ldots,s_{r})$. Let $G_{\fs, \bu}$ be defined in \eqref{E:Explicit t-moduleCMPL} and $\bv_{\fs, \bu}$ be defined in \eqref{E:v_s,u}. Let $\bv:= \bv_{\fs, \bu}$. For this $\bv$, we let $a(t)$, $n$, $Z_{n}$ and $\bv_{n}$ be given as in Theorem \ref{Theorem1}  by taking $n$ sufficiently large and divisible by $\lcm(d_{1}, \dots, d_{r})$ so that all the properties of Theorem \ref{Theorem1} hold.
We claim the following:
\begin{itemize}
\item[$\bullet$] $\log_{G_{\fs,\bu}}\left( [v(t)^{n}] \bv_{n} \right)=\partial[v(t)^{n}] \log_{G_{\fs,\bu}}(\bv_{n})$.
\item[$\bullet$] $\log_{G_{\fs,\bu}}\left([a(t)] \bv_{n} \right)=\partial[a(t)] \log_{G_{\fs,\bu}}(\bv_{n})$. 
\end{itemize}
We mention that although we have the functional equations of $\log_{G_{\fs, \bu}}$ as formal power series, one can not argue directly that the two identities above follow from the functional equations as $\log_{G_{\fs, \bu}}$ is not entire.

Assume the claim first. Then  we have
\begin{eqnarray*}
\partial [v(t)^{n}] \log_{G_{\fs, \bu}}([a(t)] \bv_{n})
&=& \partial [v(t)^{n}] \partial [a(t)] \log_{G_{\fs, \bu}}(\bv_{n}) \\
&=& \partial [a(t)] \log_{G_{\fs, \bu}}([v(t)^{n}] \bv_{n}) \\
&=& \partial [a(t)] \log_{G_{\fs, \bu}}(\bv_{\fs, \bu}),
\end{eqnarray*}
where the first and second equalities follow from the claim. By Theorem~\ref{T: CMSPL as log} and comparing the $\wt(\fs)$-th coordinates of the both sides, we have
\begin{align*}
\Lis_{\bf{n}}({\bf{w}}) = \Lis_{\tilde{\fs}}(\tilde{\bu}) &= \dfrac{(-1)^{r-1}}{a(\theta)} \times \left(\wt(\fs)\textrm{-th coordinate of} \  \partial[a(t)] \log_{G_{\fs, \bu}}(\bv_{\fs,\bu}) \right) \\
&= \dfrac{(-1)^{r-1} v^{n}}{a(\theta)} \times \left(\wt(\fs)\textrm{-th coordinate of} \ \log_{G_{\fs, \bu}}([a(t)] \bv_{n}) \right).
\end{align*}
Here we used the fact that the $\wt(\fs)$-th component of the $\wt(\fs)$-th row of $\partial [b(t)]$ is $b(\theta)$ and the other components are zero for each $b(t) \in \Fq[t]$.
By Theorem~\ref{Theorem1} we have that $\bigl|\!\bigl| [a(t)] \bv_{n} \bigr|\!\bigr|_{v} < 1 $. Since $\| \bu \|_{v} = \| {\bf{w}} \|_{v} \leq1$, by putting $\bx:=[a(t)]\bv_{n}$ into Theorem~\ref{T: Chen} we see that the first coordinate of each CMSPL appearing in the formula of Theorem~\ref{T: Chen} has the $v$-adic absolute value strictly less than one and hence the right hand side of the equation above is in $\ocLConv_{\infty}$.  

Now, we prove the claim above. We first recall that $\cD_{G_{\fs, \bu}}$ is the domain on which $\exp_{G_{\fs, \bu}}$ is an isometry. Since $Z_{n} \in \mathcal{D}_{G_{\fs,\bu}}$ and $\bv_{n} := \exp_{G_{\fs, \bu}}(Z_{n})$, we have $\log_{G_{\fs, \bu}}(\bv_{n}) = Z_{n}$. It follows that we obtain the first desired identity
\[ \log_{G_{\fs,\bu}}\left([v(t)^{n}] \bv_{n} \right)
= \log_{G_{\fs,\bu}}\left(\bv \right)
= \partial[v(t)^{n}] Z_{n}
= \partial[v(t)^{n}]\log_{G_{\fs,\bu}}\left(\bv_{n} \right),
\]
where the second equality comes from the definition $Z_{n} := \partial[v(t)^{n}]^{-1} \log_{G_{\fs, \bu}}(\bv)$.
Since
\[ \log_{G_{\fs,\bu}} \left([a(t)]\bv_{n} \right), \ \partial[a(t)] \log_{G_{\fs,\bu}}(\bv_{n}) \] belong to $\cD_{G_{\fs, \bu}}$, on which $\exp_{G_{\fs,\bu}}$ is an isometry,
the second desired identity
\[ \log_{G_{\fs,\bu}} \left([a(t)]\bv_{n} \right) = \partial[a(t)] \log_{G_{\fs,\bu}}(\bv_{n}) \]
follows from the functional equations of $\exp_{G_{\fs, \bu}}$ and its entireness:
\[ \exp_{G_{\fs, \bu}} \left( \log_{G_{\fs,\bu}} \left( [a(t)] \bv_{n} \right) \right) = [a(t)] \bv_{n} = \exp_{G_{\fs, \bu}} \left( \partial[a(t)] \log_{G_{\fs,\bu}}(\bv_{n}) \right). \]
\end{proof}

\section{Main theorem and proof} \label{final-section}
The primary goal of this section is to prove Theorem~\ref{T: IntrodT2}. 
\subsection{Yu's sub-$t$-module theorem}

In our function field setting, we have the following analogue of W\"ustholz's theory, called Yu's sub-$t$-module theorem.
\begin{theorem}[{\cite[Thm.~0.1]{Yu97}}]\label{Thm: sub-t-module thm}
Let $G$ be a regular t-module defined over $\ok$. Let $Z$ be a vector in $\Lie G(\CC_{\infty})$ such that $\exp_{G}(Z)\in G(\ok)$. Then the smallest linear subspace in $\Lie G(\CC_{\infty})$ defined over $\ok$, which is invariant under $\partial[t]$ and contains $Z$, is the tangent space at the origin of a sub-$t$-module $H$ of $G$ over $\ok$. 
\end{theorem}

Here, a sub-$t$-module of $G$ over $\ok$ is a connected algebraic subgroup of $G$ defined over $\ok$ so that it is invariant under the $[t]$-action. The following lemma plays a crucial role so that we can apply Yu's sub-$t$-module theorem appropriately to prove Theorem~\ref{T: phi-v linear}.

\begin{lemma} \label{lemma-small-v-adically}
Let $\fs \in \NN^{r}$ be an index and $\bu \in \ok^{r}$ such that $\tilde{\bu} \in \ok^{r} \cap \DDDef_{\tilde{\fs}, v}$ defined in \eqref{E: Dr,v}. For any point $\bx \in G_{\fs, \bu}(\CC_{v})$ with $\| \bx \|_{v} < 1$ and any $\epsilon > 0$, there exists $n \in \ZZ_{\geq 0}$ such that $\bigl| \! \bigl| [v(t)^{n}] \bx \bigr| \! \bigr|_{v} < \epsilon$.
\end{lemma}

\begin{proof}
For each $\epsilon > 0$, we set
\[
\fa_{\epsilon} := \{ x \in \cO_{\CC_{v}} | \ |x|_{v} < \epsilon \}.
\]
Since $G_{\fs, \bu}$ is defined over $\ok \cap \cO_{\CC_{v}}$, it is clear that the $\Fq[t]$-action $[-]$ on $G_{\fs, \bu}$ induces an $\Fq[t]$-action on $(\cO_{\CC_{v}} / \fa_{\epsilon})^{d}$ via $[-]$ where $d$ is given in \eqref{E: d}, and without confusion we denote by $G_{\fs, \bu}(\cO_{\CC_{v}} / \fa_{\epsilon})$ for the $\Fq[t]$-module $\left( (\cO_{\CC_{v}} / \fa_{\epsilon} \right)^{d}, [-])$. Note that by definition we have the following equivalence
\[
\bigl| \! \bigl| [v(t)^{n}] \bx \bigr| \! \bigr|_{v} < \epsilon \ \Longleftrightarrow \ [v(t)^{n}] (\bx \bmod {\fa_{\epsilon}}) = 0 \ \mathrm{in} \ G_{\fs, \bu}(\cO_{\CC_{v}} / \fa_{\epsilon}).
\]

We prove the lemma by induction on the depth $r = \dep(\fs)$. When $r = 1$ and $\fs = (s)$, by \cite[Proposition 1.6.1]{AT90}, we have $f_{ij}(X) \in v X A[X]$ ($1 \leq i, j \leq r$) for $b(t) = v(t)^{s}$ in Lemma \ref{L: P-equations}. Therefore we have

\begin{equation}\label{E: v(t)x}
\bigl| \! \bigl| [v(t)^{s}] \bx \bigr| \! \bigr|_{v} \leq \max\{ \| \bx \|_{v}^{q_{v}}, \| \bx \|_{v} / q_{v} \} = \left\{ \begin{array}{ll} \| \bx \|_{v}^{q_{v}} & ( \mathrm{if} \ \| \bx \|_{v} \geq q_{v}^{- 1 / (q_{v} - 1)} ) \\ \| \bx \|_{v} / q_{v} & (\mathrm{if} \ \| \bx \|_{v} \leq q_{v}^{- 1 / (q_{v} - 1)} )\end{array} \right..
\end{equation}
We set $\bx_{i} := [v(t)^{is}] \bx$. If $\| \bx_{i} \|_{v} > q_{v}^{- 1 / (q_{v} - 1)}$ for all $i \in \ZZ_{\geq 0}$, then $q_{v}^{- 1 / (q_{v} - 1)} < \| \bx_{i} \|_{v} \leq \| \bx \|_{v}^{q_{v}^{i}}$ for all $i \in \ZZ_{\geq 0}$, where the second inequality comes from \eqref{E: v(t)x}. Since $\| \bx \|_{v} < 1$, we have a contradiction. Therefore there exists $i_{0} \in \ZZ_{\geq 0}$ such that $\| \bx_{i_{0}} \|_{v} \leq q_{v}^{- 1 / (q_{v} - 1)}$.
Then we have
\[
\| \bx_{i} \|_{v} \leq \| \bx_{i_{0}} \|_{v} / q_{v}^{i - i_{0}} \ (i \geq i_{0})
\]
and hence $\| \bx_{i} \|_{v} \to 0 \ (i \to \infty)$.

Next, let $r \geq 2$ and assume that the lemma holds for $G' := G_{(s_{2}, \dots, s_{r}), (u_{2}, \dots, u_{r})}$. Let
\[
\pi := \left( (x_{1, 1}, \dots, x_{1, d_{1}}, x_{2, 1}, \dots, x_{2, d_{2}}, \dots)^{\tr} \mapsto (x_{2, 1}, \dots, x_{2, d_{2}}, \dots)^{\tr} \right) \colon G_{\fs, \bu} \twoheadrightarrow G'
\]
be the natural projection. Then we have the following exact sequence of $\Fq[t]$-modules
\[
\xymatrix{
0 \ar@{->}[r] & \bC^{\otimes \wt(\fs)}(\cO_{\CC_{v}} / \fa_{\epsilon}) \ar@{->}[r] & G_{\fs, \bu}(\cO_{\CC_{v}} / \fa_{\epsilon}) \ar@{->}[r]^{\pi_{\epsilon}} & G'(\cO_{\CC_{v}} / \fa_{\epsilon}) \ar@{->}[r] & 0,
}\]
which is induced from the short exact sequence of $t$-modules
\[
\xymatrix{
0 \ar@{->}[r] & \bC^{\otimes \wt(\fs)} \ar@{->}[r] & G_{\fs, \bu} \ar@{->}[r]^{\pi} & G' \ar@{->}[r] & 0.
}\]
By the induction hypothesis, there exists $n' \in \ZZ_{\geq 0}$ such that
\[
\pi_{\epsilon} \left( [v(t)^{n'}] (\bx \bmod \fa_{\epsilon}) \right) = [v(t)^{n'}] \left( \pi_{\epsilon}(\bx \bmod {\fa_{\epsilon}})\right) = 0
\]
and hence $[v(t)^{n'}] (\bx \bmod {\fa_{\epsilon}}) \in \ker \pi_{\epsilon} = \bC^{\otimes \wt(\fs)}(\cO_{\CC_{v}} / \fa_{\epsilon})$. By the argument of the depth one case, there exists $n_{1} \in \ZZ_{\geq 0}$ such that
\[
[v(t)^{n_{1}}] \left( [v(t)^{n'}] (\bx \bmod {\fa_{\epsilon}}) \right) = 0,
\]
whence deriving
\[
[v(t)^{n_{1} + n'}] (\bx \bmod {\fa_{\epsilon}}) = 0.
\]
\end{proof}

We recall the notion of tractable coordinates introduced by Brownawell-Papanikolas, which is convenient for us when applying Yu's sub-$t$-module theorem.
\begin{definition} \label{def tractable}
Let $G = (\GG_{a}^{d}, [-])$ be a $t$-module over $\ok$ and let $\bX = (X_{1}, \dots, X_{d})^{\tr}$ be the coordinates of $\Lie G$. The $i$-th coordinate $X_{i}$ is called {\it tractable} if the $i$-th coordinate of $\partial[a(t)] \bX$ is equal to $a(\theta) \cdot X_{i}$ for each $a(t) \in \Fq[t]$.
\end{definition}
By the definition of the $t$-module $G_{\fs, \bu}$ in \eqref{E:Explicit t-moduleCMPL}, the $(d_{1} + \cdots + d_{i})$-th coordinate of $\Lie G_{\fs, \bu}$ is tractable for each $1 \leq i \leq r$. In particular, the $\wt(\fs)$-th coordinate of $\Lie G_{\fs, \bu}$ is tractable.

\begin{theorem}\label{T: phi-v linear} The map
\[ \phi_{v}:=\left( \Lis_{\fs}(\bu)\mapsto \Lis_{\fs}(\bu)_{v} \right): \ocLDef_{\infty} \twoheadrightarrow \ocLDef_{v} \]
is a well-defined $\ok$-linear map. 
\end{theorem}
\begin{proof} Suppose that we have $\alpha_{0}+\sum_{i=1}^{m} \alpha_{i} \Lis_{\fn_{i}}(\bw_{i})=0$ for $\alpha_{0},\alpha_{1},\ldots,\alpha_{m}\in \ok$ (not all zero), $\bn_{i} \in \bigcup_{r > 0} \NN^{r}$, $\bw_{i} \in \DDDef_{\fn_{i}, \infty}$. Our aim is to show that 
\[\alpha_{0}+\sum_{i=1}^{m} \alpha_{i} \Lis_{\fn_{i}}(\bw_{i})_{v}=0 .\]

We set $\fs_{i} := \tilde{\fn}_{i}$ and $\bu_{i} := \tilde{\bw}_{i}$ and define the $t$-module 
\[G:=\GG_{a} \oplus \bigoplus_{i=1}^{m} G_{\fs_{i}, \bu_{i}} \]
with diagonal $t$-action, where $\GG_{a}$ is referred to the trivial $t$-module with exponential and logarithm maps given by the identity map $z\mapsto z$. 

Let $X_{0}$ be the coordinate of $\Lie \GG_{a}$ and $\mathbf{X}_{i}$ be the coordinates of $\Lie G_{\fs_{i},\bu_{i}}$ for $i=1,\ldots,m$. Put $j_{i}:=\wt(\fs_{i})$ and define $X_{i j_{i}}$ to be the $j_{i}$-th coordinate of $\mathbf{X}_{i}$ which is tractable in $\Lie G_{\fs_{i},\bu_{i}}$ for $i=1,\ldots,m$. So $\bX:=\left(X_{0},\mathbf{X}_{1}^{\tr},\ldots,\mathbf{X}_{m}^{\tr} \right)^{\tr}$ are coordinates of $\Lie G$ and $\left\{X_{0},X_{1 j_{1}}\ldots, X_{ m j_{m}}\right\}$ are tractable coordinates. 

 Put $\bv:=(1)\oplus \bigoplus_{i=1}^{m} \bv_{\fs_{i}, \bu_{i}}\in G(\ok)$. By Theorem \ref{proposition-log-converge}, there exists a nonzero polynomial $a(t) \in \Fq[t]$ such that $\bigl| \! \bigl| [a(t)] \bv \bigr| \! \bigr|_{v} < 1$. By Theorem~\ref{T: CMSPL as log} we have that $\log_{G}$ converges $\infty$-adically (resp.\ $v$-adically) at $\bv$ (resp.\ $[a(t)] \bv$). Note that  $\partial[a(t)] \log_{G}\left(\bv\right)$ is the column vector whose entries are the concatenation of $a(\theta)$ and the column vectors \[\partial[a(t)]\log_{G_{\fs_{1},\bu_{1}}}(\bv_{\fs_{1},\bu_{1}}), \ldots, \partial[a(t)]\log_{G_{\fs_{m},\bu_{m}}}(\bv_{\fs_{m},\bu_{m}}),\] and $ \log_{G}\left([a(t)]\bv\right)_{v}$ is the column vector whose entries are the concatenation of $a(\theta)$ and the column vectors \[\log_{G_{\fs_{1},\bu_{1}}}([a(t)]\bv_{\fs_{1},\bu_{1}})_{v}, \ldots, \log_{G_{\fs_{m},\bu_{m}}}([a(t)]\bv_{\fs_{m},\bu_{m}})_{v}.\] Furthermore, by Theorem~\ref{T: CMSPL as log} the value
$(-1)^{\dep(\fn_{i}) - 1} a(\theta) \Lis_{\fn_{i}}(\bw_{i})$ (resp. $(-1)^{\dep(\fn_{i}) - 1} a(\theta) \Lis_{\fn_{i}}(\bw_{i})_{v}$) occurs as the $j_{i}$-th coordinate of $\partial[a(t)] \log_{G_{\fs_{i},\bu_{i}}}(\bv_{\fs_{i},\bu_{i}})$ (resp.~$\log_{G_{\fs_{i},\bu_{i}}}([a(t)]\bv_{\fs_{i},\bu_{i}})_{v}$) for $i=1,\ldots,m$.

 Let $V$ be the smallest $\ok$-linear subvariety of $\Lie G$ for which
\begin{enumerate}
\item[$\bullet$] $V(\CC_{\infty})$ contains the vector $\partial[a(t)] \log_{G}(\bv)$.
\item[$\bullet$] $V$ is invariant under the $\partial [t]$-action.
\end{enumerate}
By Yu's sub-$t$-module theorem, we have $V=\Lie H$ for some sub-$t$-module $H$ of $G$ defined over $\ok$. We note that the hyperplane 
\[ \alpha_{0} X_{0} + (-1)^{\dep(\fn_{1}) - 1} \alpha_{1} X_{1j_1} + \cdots + (-1)^{\dep(\fn_{m}) - 1} \alpha_{m} X_{mj_m} =0  \]
is a $\ok$-linear subvariety of $\Lie G$ and contains the vector $\partial[a(t)] \log_{G}(\bv)$ as a $\CC_{\infty}$-valued point and is invariant under the $\partial [t]$-action. It follows from the definition of $V$ that 
\[V=\Lie H \subset \left\{ \alpha_{0} X_{0} + (-1)^{\dep(\fn_{1}) - 1} \alpha_{1} X_{1j_{1}} + \cdots + (-1)^{\dep(\fn_{m}) - 1} \alpha_{m} X_{mj_{m}} = 0 \right\}.   \]
As $\partial[a(t)] \log_{G}(\bv) \in \Lie H(\CC_{\infty}) \subseteq \Lie G(\CC_{\infty})$, we have that
\[
[a(t)] \bv = \exp_{G}\left( \partial[a(t)] \log_{G}(\bv) \right) = \exp_{H}\left( \partial[a(t)] \log_{G}(\bv) \right) \in H(\ok).
\]
By putting $\bx := [a(t)] \bv$ in Lemma \ref{lemma-small-v-adically}, there exisits $n \in \ZZ_{\geq 0}$ such that $[v(t)^{n}] [a(t)] \bv \in H(\ok)$ is $v$-adically small. Then by the same arguments of~\cite[Thm.~6.4.1]{CM19b}, we have that
\[
\log_{G} ([v(t)^{n} a(t)] \bv)_{v} = \log_{H} ([v(t)^{n} a(t)] \bv)_{v} \in \Lie H(\CC_{v}),
\]
and hence this vector is a $\CC_{v}$-valued point of the hyperplane above. That is, the desired linear relation holds.
\end{proof}

\subsection{Proof of Theorem~\ref{T: IntrodT2}}
Now we give a proof of Theorem~\ref{T: IntrodT2}. We first note that by~\eqref{E: MZV-formula} and~\eqref{E: v-adic MZV-formula} we have $\overline{\mathcal{Z}}_{\infty}\subset \ocLDef_{\infty}$, $\overline{\mathcal{Z}}_{v}\subset \ocLDef_{v}$, $\phi_{v}(\overline{\cZ}) = \overline{\cZ}_{v}$ and $\phi_{v}(\zeta_{A}(\fs)) = \zeta_{A}(\fs)_{v}$ for each index $\fs$. We note  that Theorem~\ref{theorem:linear-comb} implies $\ocLDef_{v}=\ocLConv_{v}$ as explained in~\eqref{E: DiagramIntrod}. Therefore, we have the following commutative diagram:
\begin{equation*}\label{E: DiagramLis}
\xymatrix{
\overline{\cZ} \ar@{->>}[d]_{\phi_{v}|_{\overline{\cZ}}} \ar@{^{(}->}[rr] & & \ocLDef_{\infty} \ar@{->>}[d]^{\phi_{v}} \ar@{=}[rr] & & \ocLConv_{\infty} & & \hcmspl{0} \ar@{->>}[ll]_{\cLis(-)} \ar@{->>}[lld]^{\cLis(-)_{v}} \\
\overline{\cZ}_{v} \ar@{^{(}->}[rr] & & \ocLDef_{v} & & \ocLConv_{v} \ar@{=}[ll] & &
}
\end{equation*}
where the commutativity $\phi_{v} \circ \cLis(-) = \cLis(-)_{v}$ comes from the definitions of $\phi_{v}$, $\cLis(-)$ and $\cLis(-)_{v}$.

Note that $\ocLDef_{\infty}$ and $\ocLDef_{v}$ form $\ok$-algebras by Proposition~\ref{stuffle-product-formula}. We first show that the map $\phi_{v}$ is a $\ok$-algebra homomorphism. Indeed, for each $x, x' \in \ocLDef_{\infty}$, let $w, w' \in \hcmspl{0}$ such that $\cLis(w) = x$ and $\cLis(w') = x'$. Since $\cLis(-)$ and $\cLis(-)_{v}$ are multiplicative in the sense of Proposition \ref{stuffle-product-formula}, we have
\begin{align*}
\phi_{v}(x \cdot x') &= \phi_{v}(\cLis(w) \cdot \cLis(w')) = \phi_{v}(\cLis(w \star w')) = \cLis(w \star w')_{v} \\
&= \cLis(w)_{v} \cdot \cLis(w')_{v} = \phi_{v}(\cLis(w)) \cdot \phi_{v}(\cLis(w')) =\phi_{v}(x) \cdot \phi_{v}(x').
\end{align*}

Next we show that $v$-adic MZV's satisfy the $q$-shuffle relations. Let $\fs$ and $\fs'$ be two indices and let $f_{j} \in \FF_{p}$ and $\fs_{j} \in \NN^{\dep(\fs_{j})}$ be as in \eqref{E:sum-shuffle}. Then we have
\begin{align*}
\zeta_{A}(\fs)_{v} \cdot \zeta_{A}(\fs')_{v} &= \phi_{v}(\zeta_{A}(\fs)) \cdot \phi_{v}(\zeta_{A}(\fs')) = \phi_{v}(\zeta_{A}(\fs) \cdot \zeta_{A}(\fs')) = \phi_{v}\left( \sum_{j} f_{j} \zeta_{A}(\fs_{j}) \right) \\
&= \sum_{j} f_{j} \phi_{v}(\zeta_{A}(\fs_{j})) = \sum_{j} f_{j} \zeta_{A}(\fs_{j})_{v}. 
\end{align*}
Therefore  Theorem \ref{T: IntrodT2} $(1)$ holds.
In particular, $\overline{\cZ}_{v}$ is closed under the product, and hence $\phi_{v}|_{\overline{\cZ}}$ is a $\ok$-algebra homomorphism. This shows Theorem \ref{T: IntrodT2} $(2)$.

\begin{remark}
In the proof above, we verify the identity
\[ \ocLDef_{v} = \ocLConv_{v}, \]
which generalizes~\cite[Cor.\ 3.2.11]{Chen20} for $v$-adic  CMSPL's at integral points.
\end{remark}

\begin{remark}
The definition of $\zeta_{A}(\fs)_{v}$ in \eqref{E: v-adic MZV-formula} a priori depends on the extensions of the $v$-adic CMSPL's $\Lis_{\fs_{\ell}}$ to $\DDDef_{\fs_{\ell}, v}$.
However, by Theorems \ref{theorem:linear-comb} and \ref{T: phi-v linear}, $\zeta_{A}(\fs)_{v}$ is the image of $\zeta_{A}(\fs)$ via the homomorphism $\phi_{v} \colon \ocLConv_{\infty} \to \ocLConv_{v}$ which is `canonical' once we fix embeddings $\ok \hookrightarrow \CC_{\infty}$ and $\ok \hookrightarrow \CC_{v}$ over $k$.
Note that the definition of $\zeta_{A}(\fs)_{v}$ does not depend on the choice of such embeddings. Indeed, 
if we take another pair of embeddings with $\phi_{v}'$ as the corresponding homomorphism, then the equality $\phi_{v}(\zeta_{A}(\fs)) = \zeta_{A}(\fs)_{v} = \phi_{v}'(\zeta_{A}(\fs))$ is still valid in $k_{v}$.
\end{remark}

\subsection{An example}\label{Subsec: Example}
    We provide an example of direct computations for Theorem~\ref{T: IntrodT2}~(1). Recall Huei-Jeng Chen's explicit formula \cite{Ch15} for the product of two Carlitz zeta values
        \begin{align}\label{E: H-J formula}
            \zeta_A(r)\zeta_A(s)&=\zeta_A(r,s)+\zeta_A(s,r)+\zeta_A(r+s) \\
            &+\sum_{i+j=r+s,~(q-1)\mid j} \left[ (-1)^{s-1}\binom{j-1}{s-1}+(-1)^{r-1}\binom{j-1}{r-1} \right] \zeta_A(i,j). \nonumber
        \end{align}

    Now, we compute the simplest case $\zeta_{A}(1)_{\theta}\cdot \zeta_{A}(1)_{\theta}$ for $q=2^{\ell}$  ($\ell\in \NN$).  Note that even for this simplest case, it still involves heavy computation on the explicit action of Carlitz tensor powers. In fact, to verify the validity of the $q$-shuffle product of $v$-adic MZV's by direct computations seems impracticable.

    We start by specializing $r=1$ and $s=1$ in Huei-Jeng Chen's formula and we get \[ \zeta_A(1) \cdot \zeta_A(1) = \zeta_A(2)\] as the characteristic of the base field is $2$. Note that this relation also follows from the definition of MZV's directly. To verify \[ \zeta_{A}(1)_{\theta} \cdot \zeta_{A}(1)_{\theta} = \zeta_{A}(2)_{\theta}, \] we recall that the Anderson-Thakur polynomials \cite{AT90,AT09} $H_{n-1}$ is equal to $1$ if $1 \leq n \leq q$. In this case, we have $$\zeta_A(n)_\theta=\Li_n^\star(1)_\theta$$ (see \cite[Sec.~5,~Sec.~6]{CM19b} for details). So our task is to calculate $\Li_n^\star(1)_\theta$. By definition, $\Li_n^\star(1)_\theta$ is given by
\begin{align*}
\frac{1}{\theta^n-1}\times \left( n\textrm{-th coordinate of }\log_{\bC^{\otimes n}}([t^n-1](1) )_{\theta} \right),
\end{align*}
where $\bC^{\otimes n}$ stands for the $n$-th tensor power of the Carlitz module defined in \eqref{def-carlitz}. One can show that \cite[Ex.~3.2.12]{Chen20} $$[t^n-1](1)=(\binom{n}{1}\theta,\binom{n}{2}\theta^2,\dots,\binom{n}{n-1}\theta^{n-1},\theta^n)^\tr$$ and consequently $$\Li_n^\star(1)_\theta= \frac{1}{\theta^n-1}\left(\Li_n^\star(\theta^n )_\theta + \sum_{1 \leq j < n} \ \sum_{i=0}^{j} (-1)^{i+j}\binom{j}{i}\theta^i \Li_n^\star(\binom{n}{j}\theta^{n-i})_\theta \right)$$ by using \cite[Thm.~3.2.9]{Chen20}. In particular, we derive that $$\Li_1^\star(1)_\theta=\frac{1}{\theta-1}\Li_1^\star(\theta)_\theta$$ and $$\Li_2^\star(1)_{\theta}=\frac{1}{\theta^2-1}\Li_2^\star(\theta^2)_\theta$$ because of characteristic $2$. Then it is clear to see that
    \begin{align*}
        \zeta_A(1)_\theta \cdot \zeta_A(1)_\theta&=\Li_1^\star(1)_\theta \cdot \Li_1^\star(1)_\theta=\frac{1}{(\theta-1)^2}\Li_1^\star(\theta)_\theta \cdot \Li_1^\star(\theta)_\theta\\
        &=\frac{1}{\theta^2-1}\Li_2^\star(\theta^2)_\theta=\Li_2^\star(1)_\theta=\zeta_A(2)_\theta.
    \end{align*}

\subsection{A conjecture}\label{Subsec: Questions} In what follows, we conjecture that the kernel of the $\ok$-algebra homomorphism in Theorem~\ref{T: IntrodT2} is generated by $\zeta_{A}(q-1)$.
 \begin{conjecture}\label{Conj: kernel}
 For any finite place $v$ of $k$, we have the following $\ok$-algebra isomorphism
 \[ \overline{\mathcal{Z}}/(\zeta_{A}(q-1))\cong \overline{\mathcal{Z}}_{v}  .\]
 \end{conjecture}
The conjecture above would imply the following important consequences:
\begin{itemize}
\item[(i)] $\overline{\mathcal{Z}}_{v} \cong \overline{\mathcal{Z}}_{v'}$ for any finite places $v, v'$ of $k$. 
\item[(ii)] $\overline{\mathcal{Z}}_{v}$ is a graded $\ok$-algebra (graded by weights) defined over $k$.
\end{itemize}
Note that in~\cite[Thm.~2.2.1]{C14}, the first author of the present paper showed that $\overline{\mathcal{Z}}$ forms a graded $\ok$-algebra (graded by weights) that is defined over $k$. That is:
\begin{itemize}
\item $\infty$-adic MZV's of different weights are linearly independent over $\ok$.
\item $k$-linear independence of $\infty$-adic MZV's implies $\ok$-linear independence.
\end{itemize}
So the statement (ii) above is the $v$-adic analogue of~\cite[Thm.~2.2.1]{C14} for $\infty$-adic MZV's. We mention that in the case of $p$-adic MZV's, one has Furusho-Yamashita's conjecture~\cite[Conj.~5]{Ya10} asserting that nonzero $p$-adic MZV's of different weights are linearly independent over $\QQ$, and this is the $p$-adic analogue of Goncharov's direct sum conjecture~\cite{Gon97} for real-valued MZV's.

Now, we consider the $\overline{k}$-subalgebra $\overline{\mathcal{Z}}^{1}\subset \overline{\mathcal{Z}}$ generated by the $\infty$-adic single zeta values, namely the $\infty$-adic MZV's of depth one, and the $\overline{k}$-subalgebra $\overline{\mathcal{Z}}^{1}_v\subset \overline{\mathcal{Z}}_v$ generated by the $v$-adic single zeta values. When we restrict the $\overline{k}$-algebra homomorphism given in Theorem~\ref{T: IntrodT2} to $\overline{\mathcal{Z}}^{1}$, we obtain the surjective $\ok$-algebra homomorphism
 \[ \overline{\mathcal{Z}}^{1}\twoheadrightarrow \overline{\mathcal{Z}}_{v}^{1} .\] 
 
 It is shown in \cite{CY07} that all the algebraic relations among $\infty$-adic single zeta values are generated by Euler-Carlitz relations, namely $\zeta_{A}((q-1)n)/\zeta_{A}(q-1)^{n}\in k$ for $n\in \NN$, and the $p$-th power relations, ie., $\zeta_{A}(pn)=\zeta_{A}(n)^{p}$ for $n\in \NN$. Recall that by~\cite{Go79} we have the trivial zeros $\zeta_{A}((q-1)n)_{v}=0$ for $n\in \NN$. Chang-Yu's conjecture in \cite[p.~323]{CY07} asserts that all the algebraic relations among Goss' $v$-adic zeta values come from the trivial zeros above and the $p$-th power relations, and hence it implies that the kernel of the $\overline{k}$-algebra homomorphism $\overline{\mathcal{Z}}^{1}\twoheadrightarrow\overline{\mathcal{Z}}_v^{1}$ is generated by $\zeta_{A}(q-1)$ in $\overline{\mathcal{Z}}^{1}$. So Chang-Yu's conjecture matches with the phenomenon of Conjecture~\ref{Conj: kernel}.

\end{document}